\documentclass[10pt,a4paper]{amsart}
\usepackage{amsmath,amsthm,amssymb,stmaryrd,latexsym,a4wide,epic,eepic,bbm,mathrsfs,amsfonts,cases,yfonts,indentfirst,bbold}
\usepackage{dsfont,color,mathdots}

\usepackage{amsmath,amssymb}

\usepackage{amsfonts}
\usepackage{tipa}

\usepackage{CJK, CJKnumb}
\usepackage{color}      
\usepackage{indentfirst}        
\usepackage{latexsym, bm}        

\usepackage{graphicx}
\usepackage{cases}
\usepackage[all]{xypic}
\usepackage{fancyhdr}
\usepackage{pifont}
\newtheorem{theorem}{Theorem}
\newtheorem{defi}[theorem]{Definition}
\newtheorem{lemma}[theorem]{Lemma}
\newtheorem{coro}[theorem]{Corollary}
\newtheorem{proposition}[theorem]{Proposition}

\newtheorem{remark}[theorem]{Remark}
\usepackage[all]{xy}
\usepackage[active]{srcltx}
\usepackage{enumerate}
\numberwithin{equation}{section}

\begin{document}
	
	\title[$RLL$-REALIZATION OF TWO-PARAMETER QUANTUM
	AFFINE ALGEBRA IN TYPE $C_n^{(1)}$]{$RLL$-REALIZATION OF TWO-PARAMETER QUANTUM AFFINE ALGEBRA IN TYPE $C_n^{(1)}$}
	
	\author[X. Zhong]{Xin Zhong}
	\address{School of Mathematical Sciences, MOE Key Laboratory of Mathematics and Engineering Applications \& Shanghai Key Laboratory of PMMP, East China Normal University, Shanghai 200241, China}
		\email{52215500002@stu.ecnu.edu.cn}
		
	\author[N. Hu]{Naihong Hu}
	\address{School of Mathematical Sciences, MOE Key Laboratory of Mathematics and Engineering Applications \& Shanghai Key Laboratory of PMMP, East China Normal University, Shanghai 200241, China}
	\email{nhhu@math.ecnu.edu.cn}
	
\author[N. Jing]{Naihuan Jing$^*$}
\address{Department of Mathematics, North Carolina State University, Raleigh, NC 27695, USA}
\email{jing@ncsu.edu}
	
	\thanks{$^*$Corresponding author.
		The paper is supported by the NNSFC (Grants No.
		12171155, 12171303) and in part by the Science and Technology Commission of Shanghai Municipality (No. 22DZ2229014).}
	\begin{abstract}
		We establish an explicit correspondence  between the Drinfeld current algebra presentation for
		the two-parameter quantum affine algebra $U_{r, s}(\mathrm{C}_n^{(1)})$ and the $R$-matrix realization \'a la Faddeev, Reshetikhin and Takhtajan.
	\end{abstract}
	
	\maketitle
	
	
	\section{Introduction}
The quantum Yang-Baxter equation (YBE) arose as a common symmetry in various physical models in quantum field theory and statistical physics \cite{B, Y} and provides
the common governing equation for quantum groups. The quantum affine algebras $U_q(\widehat{\mathfrak{g}})$ are the quantum enveloping algebra of
the affine Kac-Moody algebras $U(\widehat{\mathfrak g})$ introduced by Drinfeld and Jimbo using Chevalley generators and Serre relations. There are two other
presentations of $U_q(\widehat{\mathfrak{g}})$: the Faddeev-Reshetikhin-Takhtajan formulism or the $R$-matrix realization given by
Reshetikhin and Semenov-Tian-Shansky \cite{RS}
using the spectra-parameter-dependent Yang-Baxter equation in the framework of the quantum inverse scattering method \cite{FRT, TF}. 
The third realization is Drinfeld's new realization, a $q$-analogue of the loop realization of the classical affine Lie algebra \cite{D}.

The three realizations are usually believed to be equivalent for quantum affine algebras and Yangians as the first such isomorphism was announced
by Drinfeld \cite{D}. Ding and Frenkel showed the $R$-matrix realization is isomorphic to the Drinfeld realization for the quantum affine algebra in type $A$ by using the Gauss decomposition, which also provided a natural way to derive the Drinfeld realization from the $R$-matrix
realization of the quantum affine algebra \cite{DF}. Similarly Brundan and Kleshchev have proven the isomorphism for the Yangian algebra
 in type $A$ \cite{BK}.
 Jing, Liu, and Molev extended this result to types $B_n^{(1)}$, $C_n^{(1)}$, $D_n^{(1)}$ both for the Yangians and quantum affine algebras
 \cite{JLM1, JLM2, JLM3}. We remark that a correspondence between the $R$-matrix formulism and Drinfeld realization was given for the general type one quantum affine algebras in \cite{HM}, and recently
 the isomorphism was also established for the twisted quantum affine algebra in type $A_{2n-1}^{(2)}$ \cite{JZL}.

Meanwhile, two-parameter quantum groups have been revitalized by Benkart-Witherspoon \cite{BW2}, Bergeron-Gao-Hu \cite{BGH1}, and many others, etc. In 2004, Hu, Rosso, and Zhang investigated the two-parameter quantum affine algebra $U_{r, s}(\widehat{\mathfrak{sl}}_n)$ and obtained the
two-parameter version of the celebrated Drinfeld realization, and proposed for the first time the quantum affine Lyndon basis as a monomial basis in the quantum affine case \cite{HRZ}. The PBW basis theorem was proved in \cite{Ts} for loop algebras using shuffle algebras. 
The significance of the second parameter in $U_{r, s}(\widehat{\mathfrak{g}})$ was recognized \cite{FL} as related to the Tate twist in geometric realization of quantum groups. Furthermore, two-parameter quantum affine algebras in types $B_n^{(1)}$, $C_n^{(1)}$, $D_n^{(1)}$ have been investigated by Hu-Zhang etc, including the Drinfeld realizations and their vertex representations \cite{HZ}.
Bergeron-Gao-Hu have studied finite-dimensional representations of two-parameter quantum orthogonal and symplectic groups and constructed corresponding quasi-$R$-matrices and the quantum Casimir operators \cite{BGH2}. In \cite{JL2} the correspondence between
the $R$-matrix formulism and the Drinfeld realization of the quantum affine algebra
$U_{r, s}(\widehat{\mathfrak{gl}}_n)$ was established using Gauss decomposition.

In view of all these developments, a natural question is how to generalize the $RLL$-realizations of $U_q(\hat{\mathfrak g})$ to the two-parameter quantum affine algebras $U_{r,s}(\hat{\mathfrak{g}})$ for other classical types?

The first goal of this paper is to set up the correspondence between the Faddeev-Reshetikhin-Takhtajan formalism and the Drinfeld-Jimbo presentation for the two-parameter quantum algebra $U_{r, s}(\mathrm{C}_{\mathrm{n}})$ in finite type $C$. The second goal is to give a $RLL$-presentation for the quantum affine algebra $U_{r, s}(\mathrm{C}_n^{(1)})$ and provide its Drinfeld realization
using the Gauss decomposition. Compared with standard quantum affine algebra $U_{q}(\mathrm{C}_n^{(1)})$, the two parameter quantum affine algebras $U_{r, s}(\mathrm{C}_n^{(1)})$ have twice as many group-like elements, it requires more effort to
control additional terms in the quantum $R$ matrix. As a result there are more relations for the $RLL$-formulism.

The paper is organized as follows. In section 2, we start with the vector representation of type $1$ for $U_{r, s}\left(\mathrm{C}_{\mathrm{n}}\right)$, and we compute the corresponding basic quantum $R$-matrix using the Hopf algebra structure. In section 3, we establish the isomorphism between Faddeev-Reshetikhin-Takhtajan's and Drinfeld-Jimbo's definitions of $U_{r,s}\left(\mathrm{C}_n\right)$ and the Gauss decomposion in Remark 9. In section 4, we use Jimbo's method or the so-called Yang Baxterization \cite{GWX} to get the quantum affine $R$-matrix, and both arrive at the same result. To work with the quantum affine algebra in type $C^{(1)}_n$, we apply the Gauss decomposition of the generator matrices in the $R$-matrix presentation, we figure out the contribution of those extra terms in the $R$ matrix and find more relations which are different from one parameter case. After some complicated calculations together with $R$-matrix of type $A_N^{(1)}$, we find that the Drinfeld generators satisfy the required relations of the Drinfeld presentation.

	\section{PRELIMINARIES}
	\subsection{Notation and Definition}
	Let $\mathbb{K}=\mathbb{Q}\left(r, s\right)$ denote the field of rational functions with two-parameters $r, s \, (r \neq \pm s)$. Let $\mathfrak{g}$ be the finite-dimensional complex simple Lie algebra of type $C_n$ with Cartan subalgebra $\mathfrak{h}$ and Cartan matrix $A=\left(a_{i j}\right)_{i, j \in I}(I=\{1,2, \cdots, n\})$. Fix coprime
	integers $\left(d_i\right)_{i \in I}$ such that $\left(d_i a_{i j}\right)$  is symmetric. Assume $\Phi$ is a finite root system of type $C_n$ with $\Pi$ a base of simple roots. Regard $\Phi$ as a subset of a Euclidean space $\mathrm{E}=$$\mathbb{R}^n$ with an inner product $\langle\   , \ \rangle$. Let $\epsilon_1, \epsilon_2, \cdots, \epsilon_n$ denote an orthonormal basis of E. Let $\alpha_i=\epsilon_i-\epsilon_{i+1}, \alpha_n=2 \epsilon_n$ be the simple roots of the simple Lie algebra $\mathfrak{s p}_{2 n}$, $\left\{\alpha_i^{\vee}\right\}$ and $\left\{\lambda_i\right\}$, the sets of simple coroots and fundamental weights respectively. Let $Q=\bigoplus_{i=1}^n \mathbb{Z} \alpha_i$ be the root lattice, $\theta$ the highest root, and $\delta$ the primitive imaginary root. Take $\alpha_0=\delta-\theta$, then $\Pi^{\prime}=\left\{\alpha_i \mid i \in I_0=\{0,1, \cdots, n\}\right\}$ is a base of simple roots of the affine Lie algebra $C_n^{(1)}$. Let $\widehat{Q}=\bigoplus_{i=0}^n \mathbb{Z} \alpha_i$ denote the root lattice of $C_n^{(1)}$.
	
	Let $c$ be the canonical central element of the affine Lie algebra of type $C_n^{(1)}$. Define a nondegenerate symmetric bilinear form $(\ \mid \ )$ on $\mathfrak{h}^*$ satisfying
	$$
	\left(\alpha_i \mid \alpha_j\right)=d_i a_{i j}, \quad\left(\delta \mid \alpha_i\right)=(\delta \mid \delta)=0, \quad \text { for all } i, j \in I_0,
	$$
	where $\left(d_0, d_1, \cdots, d_n\right)=\left(2,1 , \cdots,1 , 2\right)$. Denote $r_i=r^{d_i}, s_i=s^{d_i}$ for $i=0,1$,$\cdots, n$.
	
	Given two sets of symbols $W=\left\{\omega_0, \omega_1, \cdots, \omega_n\right\}, W^{\prime}=\left\{\omega_0^{\prime}, \omega_1^{\prime}, \cdots, \omega_n^{\prime}\right\}$, define the structural constants matrix $\left(\left\langle\omega_i^{\prime}, \omega_j\right\rangle\right)_{(n+1) \times (n+1)}$ of type $C_n^{(1)}$ by
	$$
	\left(\begin{array}{cccccc}
		r^{2} s^{-2} & r^{-2} & 1 & \cdots & 1 & (rs)^{2} \\
		s^{2} & r s^{-1} & r^{-1} & \cdots & 1 & 1 \\
		\cdots & \cdots & \cdots & \cdots & \cdots & \cdots \\
		1 & 1 & 1 & \cdots & r s^{-1} & r^{-2} \\
		(r s)^{-2} & 1 & 1 & \cdots & s^{2} & r^{2} s^{-2}
	\end{array}\right).
	$$
	\begin{defi} The two-parameter quantum affine algebra $U_{r, s}\left(\mathrm{C}_n^{(1)}\right)$ is the unital associative algebra over $\mathbb{K}$ generated by the elements $e_j, f_j, \omega_j^{ \pm 1}, \,\omega_j^{\prime \pm 1}\left(j \in I_0\right)$, $\gamma^{ \pm \frac{1}{2}}$, $ \gamma^{\prime \pm \frac{1}{2}}$, $D^{ \pm 1}, D^{\prime \pm 1}$ subject to
	the following relations:
		
		$(\hat{C} 1)$ $\gamma^{ \pm \frac{1}{2}}, \gamma^{\prime \pm \frac{1}{2}}$ are central with $\gamma=\omega_\delta, \gamma^{\prime}=\omega_\delta^{\prime}$, such that $\gamma \gamma^{\prime}=(r s)^c$. The generators $\omega_i^{ \pm 1}, \omega_j^{\prime \pm 1}$ all commute with each other
		and
		$\omega_i \omega_i^{-1}=\omega_i^{\prime} \omega_i^{\prime-1}=1,\left[\omega_i^{ \pm 1}, D^{ \pm 1}\right]=\left[\omega_j^{\prime \pm 1}, D^{ \pm 1}\right]=$
		$\left[\omega_i^{ \pm 1}, D^{\prime \pm 1}\right]=\left[\omega_j^{\prime \pm 1}, D^{\prime \pm 1}\right]=\left[D^{\prime \pm 1}, D^{ \pm 1}\right]=0$.
		
		$(\hat{C} 2)$ For $0 \leqslant i \leqslant n$ and $1 \leqslant j<n$,
		$$
		D e_i D^{-1}=r_i^{\delta_{0 i}} e_i, \quad D f_i D^{-1}=r_i^{-\delta_{0 i}} f_i \text {, }
		$$
		$$
		\omega_j e_i \omega_j^{-1}=r_j^{\left(\epsilon_j, \alpha_i\right)} s_j^{\left(\epsilon_{j+1}, \alpha_i\right)} e_i, \quad \omega_j f_i \omega_j^{-1}=r_j^{-\left(\epsilon_j, \alpha_i\right)} s_j^{-\left(\epsilon_{j+1}, \alpha_i\right)} f_i,
		$$
	$$
	\omega_n e_j \omega_n^{-1}=r^{2\left(\epsilon_n, \alpha_j\right)} e_j, \quad \omega_n f_j \omega_n^{-1}=r^{-2\left(\epsilon_n, \alpha_j\right)} f_j,
	$$
	$$
	\omega_n e_n \omega_n^{-1}=r^{\left(\epsilon_n, \alpha_n\right)} s^{-\left(\epsilon_n, \alpha_n\right)} e_n, \quad \omega_n f_n \omega_n^{-1}=r^{-\left(\epsilon_n, \alpha_n\right)} s^{\left(\epsilon_n, \alpha_n\right)} f_n,
	$$
	$$
	\omega_0 e_j \omega_0^{-1}=r^{-2\left(\epsilon_{j+1}, \alpha_0\right)} s^{2\left(\epsilon_1, \alpha_j\right)} e_j, \quad \omega_0 f_j \omega_0^{-1}=r^{2\left(\epsilon_{j+1}, \alpha_0\right)} s^{-2\left(\epsilon_1, \alpha_j\right)} f_j,
	$$
	$$
	\omega_n e_0 \omega_n^{-1}=(r s)^2 e_0, \quad \omega_n f_0 \omega_n^{-1}=(r s)^{-2} f_0,
	$$
	$$
	\omega_0 e_n \omega_0^{-1}=(r s)^{-2} e_n, \quad \omega_0 f_n \omega_0^{-1}=(r s)^2 f_0,
	$$
	$$
	\omega_0 e_0 \omega_0^{-1}=r^{-\left(\epsilon_1, \alpha_0\right)} s^{\left(\epsilon_1, \alpha_0\right)} e_0, \quad \omega_0 f_0 \omega_0^{-1}=r^{\left(\epsilon_1, \alpha_0\right)} s^{-\left(\epsilon_1, \alpha_0\right)} f_0.
	$$	

		$(\hat{C} 3)$ For $0 \leqslant i \leqslant n$ and $1 \leqslant j<n$,
			$$
		D^{\prime} e_i D^{\prime-1}=s_i^{\delta_{0 i}} e_i, \quad D^{\prime} f_i D^{\prime-1}=s_i^{-\delta_{0 i}} f_i,
		$$
		$$
		\omega_j^{\prime} e_i \omega_j^{\prime-1}=s_j^{\left(\epsilon_j, \alpha_i\right)} r_j^{\left(\epsilon_{j+1}, \alpha_i\right)} e_i, \quad \omega_j^{\prime} f_i^{\prime} \omega_j^{\prime-1}=s_j^{-\left(\epsilon_j, \alpha_i\right)} r_j^{-\left(\epsilon_{j+1}, \alpha_i\right)} f_i,
		$$
		$$
		\omega_n^{\prime} e_j \omega_n^{\prime-1}=s^{2\left(\epsilon_n, \alpha_j\right)} e_j, \quad \omega_n^{\prime} f_j \omega_n^{\prime-1}=s^{-2\left(\epsilon_n, \alpha_j\right)} f_j \text {, }
		$$
		$$
		\omega_n^{\prime} e_n \omega_n^{\prime-1}=s^{\left(\epsilon_n, \alpha_n\right)} r^{-\left(\epsilon_n, \alpha_n\right)} e_n, \quad \omega_n^{\prime} f_n \omega_n^{\prime-1}=s^{-\left(\epsilon_n, \alpha_n\right)} r^{\left(\epsilon_n, \alpha_n\right)} f_n,
		$$
		$$
		\omega_0^{\prime} e_j \omega_0^{\prime-1}=s^{-2\left(\epsilon_{j+1}, \alpha_0\right)} r^{2\left(\epsilon_1, \alpha_j\right)} e_j, \quad \omega_0^{\prime} f_j \omega_0^{\prime-1}=s^{2\left(\epsilon_{j+1}, \alpha_0\right)} r^{-2\left(\epsilon_1, \alpha_j\right)} f_j,
		$$
		$$
		\omega_n^{\prime} e_0 \omega_n^{\prime-1}=(r s)^2 e_0, \quad \omega_n^{\prime} f_0 \omega_n^{\prime-1}=(r s)^{-2} f_0,
		$$
		$$
		\omega_0^{\prime} e_n \omega_0^{\prime-1}=(r s)^{-2} e_n, \quad \omega_0^{\prime} f_n \omega_0^{\prime-1}=(r s)^2 f_0,
		$$
		$$
		\omega_0^{\prime} e_0 \omega_0^{\prime-1}=s^{-\left(\epsilon_1, \alpha_0\right)} r^{\left(\epsilon_1, \alpha_0\right)} e_0, \quad \omega_0^{\prime} f_0 \omega_0^{\prime-1}=s^{\left(\epsilon_1, \alpha_0\right)} r^{-\left(\epsilon_1, \alpha_0\right)} f_0.
		$$	
		
		$(\hat{C} 4)$ For $i, j \in I_0$, we have:
		$$\left[e_i, f_j\right]=\frac{\delta_{i j}}{r_i-s_i}\left(\omega_i-\omega_i^{\prime}\right)$$
		
		$(\hat{C} 5)$ For all $1 \leqslant i \neq j \leqslant n$ but $(i, j) \notin\{(0, n),(n, 0)\}$ such that $a_{i j}=0$, we have:
$$
\left[e_i, e_j\right]=\left[f_i, f_j\right]=0,
$$
$$
e_n e_0=(r s)^2 e_0 e_n, \quad f_0 f_n=(r s)^2 f_n f_0.
$$

		($\hat{C} 6)$ For $1 \leqslant i \leqslant n-2$, the following $(r, s)$-Serre relations hold for $e_i'$s:
	$$
	\begin{gathered}
		e_i^2 e_{i+1}-(r+s) e_i e_{i+1} e_i+(r s) e_{i+1} e_i^2=0, \\
		e_0^2 e_1-(r+s) e_0 e_1 e_0+r s e_1 e_0^2=0, \\
		e_{i+1}^2 e_i-\left(r_{i+1}^{-1}+s_{i+1}^{-1}\right) e_{i+1} e_i e_{i+1}+\left(r_{i+1}^{-1} s_{i+1}^{-1}\right) e_i e_{i+1}^2=0, \\
		e_n^2 e_{n-1}-\left(r^{-1}+s^{-1}\right) e_n e_{n-1} e_n+\left(r^{-1} s^{-1}\right) e_{n-1} e_n^2=0,
	\end{gathered}
	$$
	$$
	\begin{aligned}
		& e_{n-1}^3 e_n-\left(r+(r s)^{\frac{1}{2}}+s\right) e_{n-1}^2 e_n e_{n-1}+(r s)^{\frac{1}{2}}\left(r+(r s)^{\frac{1}{2}}+s\right) e_{n-1} e_n e_{n-1}^2-(r s)^{\frac{3}{2}} e_n e_{n-1}^3=0, \\
		& e_1^3 e_0-\left(r^{-1}+(r s)^{-\frac{1}{2}}+s^{-1}\right) e_1^2 e_0 e_1+(r s)^{-\frac{1}{2}}\left(r^{-1}+(r s)^{-\frac{1}{2}}+s^{-1}\right) e_1 e_0 e_1^2-(r s)^{-\frac{3}{2}} e_0 e_1^3=0.
	\end{aligned}
	$$
		$(\hat{C} 7)$ For $1 \leqslant i \leqslant n-2$, the following $(r, s)$-Serre relations hold for $f_i'$s::
	$$
	\begin{gathered}
		f_{i+1} f_i^2-(r+s) f_i f_{i+1} f_i+(r s) f_i^2 f_{i+1}=0, \\
		f_1 f_0^2-(r+s) f_0 f_1 f_0+r s f_0^2 f_1=0, \\
		f_i f_{i+1}^2-\left(r_{i+1}^{-1}+s_{i+1}^{-1}\right) f_{i+1} f_i f_{i+1}+\left(r_{i+1}^{-1} s_{i+1}^{-1}\right) f_{i+1}^2 f_i=0, \\
		f_{n-1} f_n^2-\left(r^{-1}+s^{-1}\right) f_n f_{n-1} f_n+\left(r^{-1} s^{-1}\right) f_n^2 f_{n-1}=0, \\
		f_n f_{n-1}^3-\left(r+(r s)^{\frac{1}{2}}+s\right) f_{n-1} f_n f_{n-1}^2+(r s)^{\frac{1}{2}}\left(r+(r s)^{\frac{1}{2}}+s\right) f_{n-1}^2 f_n f_{n-1}-(r s)^{\frac{3}{2}} f_{n-1}^3 f_n=0, \\
		f_0 f_1^3-\left(r^{-1}+(r s)^{-\frac{1}{2}}+s^{-1}\right) f_1 f_0 f_1^2+(r s)^{-\frac{1}{2}}\left(r^{-1}+(r s)^{-\frac{1}{2}}+s^{-1}\right) f_1^2 f_0 f_1-(r s)^{-\frac{3}{2}} f_1^3 f_0=0 .
	\end{gathered}
	$$
\end{defi}
The algebra $U_{r, s}\left(\mathrm{C}_n^{(1)}\right)$  is a Hopf algebra with the coproduct $\Delta$, the counit $\varepsilon$ and the antipode $S$, defined below: for $i \in I_0$,
$$
\begin{gathered}
	\Delta\left(\gamma^{ \pm \frac{1}{2}}\right)=\gamma^{ \pm \frac{1}{2}} \otimes \gamma^{ \pm \frac{1}{2}}, \quad \Delta\left(\gamma^{\prime \pm \frac{1}{2}}\right)=\gamma^{\prime \pm \frac{1}{2}} \otimes \gamma^{\prime \pm \frac{1}{2}}, \\
	\Delta\left(D^{ \pm 1}\right)=D^{ \pm 1} \otimes D^{ \pm 1}, \quad \Delta\left(D^{\prime \pm 1}\right)=D^{\prime \pm 1} \otimes D^{\prime \pm 1}, \\
	\Delta\left(\omega_i\right)=\omega_i \otimes \omega_i, \quad \Delta\left(\omega_i^{\prime}\right)=\omega_i^{\prime} \otimes \omega_i^{\prime}, \\
	\Delta\left(e_i\right)= 1\otimes e_i+ e_i\otimes \omega_i, \quad \Delta\left(f_i\right)= \omega_i^{\prime}\otimes f_i +f_i \otimes 1 , \\
	\varepsilon\left(e_i\right)=\varepsilon\left(f_i\right)=0, \quad \varepsilon\left(\gamma^{ \pm \frac{1}{2}}\right)=\varepsilon\left(\gamma^{\prime \pm \frac{1}{2}}\right)=\varepsilon\left(D^{ \pm 1}\right)=\varepsilon\left(D^{\prime \pm 1}\right)=\varepsilon\left(\omega_i\right)=\varepsilon\left(\omega_i^{\prime}\right)=1, \\
	S\left(\gamma^{ \pm \frac{1}{2}}\right)=\gamma^{\mp \frac{1}{2}}, \quad S\left(\gamma^{\prime \pm \frac{1}{2}}\right)=\gamma^{\prime \mp \frac{1}{2}}, \quad S\left(D^{ \pm 1}\right)=D^{\mp 1}, \quad S\left(D^{\prime \pm 1}\right)=D^{\prime \mp 1}, \\
	S\left(e_i\right)=-e_i \omega_i^{-1} , \quad S\left(f_i\right)=-\omega_i^{\prime-1} f_i , \quad S\left(\omega_i\right)=\omega_i^{-1}, \quad S\left(\omega_i^{\prime}\right)=\omega_i^{\prime-1} .
\end{gathered}
$$

It is clear that $\tau: e_i \to f_i , \omega_i\to \omega_i', r\to s, s\to r$ is an anti-automorphism of $U_{r, s}(\widehat{\mathfrak g})$ over $\mathbb C$.

	\subsection{Basic quantum $R$-matrix for $U_{r, s}\left(\mathrm{C_n}\right)$}
In this section, we consider the $R$-matrix $R_{1,1}:=R_{V, V}$ for the vector representation $T_1=T_V$ of the algebra $U_{r, s}\left(\mathrm{C}_n\right)$.
The $R$-matrices for quantum affine algebras were given by Jimbo in \cite{J}, which serve as the starting point in QISM.
We now consider the generalized $(r, s)$-dependent $R$-matrix for the fundamental representation of $U_{r, s}(\mathfrak{sp}(2n))$.
Consider the $2n$-dimensional vector space $V$ over $\mathbb{K}$ with basis $\left\{v_j \mid 1 \leq j \leq 2 n\right\}$. The following defines
the vector representation of $U_{r, s}\left(\mathrm{C}_n\right)$.
	\begin{lemma}
		Let $E_{k l}$ be the $2 n \times 2 n$ matrix with $1$ in the $(k, l)$ position and $0$
		elsewhere. The vector representation $T_1$ of $U_{r, s}\left(\mathrm{C}_n\right)$  is described by the following formulas
\begin{align*}
	T_1\left(e_i\right)&=r^{-\frac{1}{2}} s^{-\frac{1}{2}} E_{i+1, i}-E_{2 n-i+1,2 n-i}, \quad T_1\left(e_n\right)=r^{-\frac{1}{2}} s^{-\frac{1}{2}} E_{n+1, n}, \\
	T_1\left(f_i\right)&=r^{-\frac{1}{2}} s^{-\frac{1}{2}} E_{i, i+1}-E_{2 n-i, 2 n-i+1}, \quad T_1\left(f_n\right)=r^{-\frac{1}{2}} s^{-\frac{1}{2}} E_{n, n+1}, \\
	T_1\left(w_i\right)&=r^{-1} E_{i, i}+s^{-1} E_{i+1, i+1}+s E_{2 n-i, 2 n-i}+r E_{2 n-i+1,2 n-i+1}+\sum_{j \neq\{i, i+1,2 n-i, 2 n-i+1\}} E_{j, j}, \\
	T_1\left(w_i'\right)&=s^{-1} E_{i, i}+r^{-1} E_{i+1, i+1}+r E_{2 n-i, 2 n-i}+s E_{2 n-i+1,2 n-i+1}+\sum_{j \neq\{i, i+1,2 n-i, 2 n-i+1\}} E_{j, j},\\
T_1\left(w_n\right)&=r s \sum_{i=1}^{n-1} E_{i, i}+r^{-1} s E_{n, n}+r s^{-1} E_{n+1, n+1}+r^{-1} s^{-1} \sum_{i=n+2}^{2 n} E_{i, i}, \\
	T_1\left(w_n'\right)&=r s \sum_{i=1}^{n-1} E_{i, i}+s^{-1} r E_{n, n}+s r^{-1} E_{n+1, n+1}+r^{-1} s^{-1} \sum_{i=n+2}^{2 n} E_{i, i} .
\end{align*}
\end{lemma}
\begin{proof}
	By straightforward calculations one checks that the preceding assignment
	defines an irreducible weight representation. For the basis vector $\mathbf{e}_1=(1,0, \cdots, 0)$ it is easy to see that $T_1\left(w_i^{-1}\right) \mathbf{e}_1=r^{\delta_{1 i}} \mathbf{e}_1$ for $1 \leq i \leq n-1$, $T_1\left(w_n^{-1}\right) \mathbf{e}_1=r^{-1} s^{-1} \mathbf{e}_1$ and $T_1\left(f_j\right) \mathbf{e}_1=0$ for $1 \leq i \leq n$. Hence $T_1$ is the type 1 representation with the highest weight $\lambda=(1,0, \cdots, 0)$ with respect to the simple roots $-\alpha_1, \cdots,-\alpha_n$.
	\begin{remark}
		If we replace $T_1\left(w_i^{-1}\right)$ with $T_1\left(w_i\right)$,  $T_1\left(w_i^{\prime-1}\right)$ with $T_1\left(w_i^{\prime}\right)$ and $T_1\left(e_i\right)$ with $T_1\left(f_i\right)$, we must change Gaussian decomposition accordingly, which ensures the correspondence between Drinfeld and Drinfeld-Jimbo realizations.
	\end{remark}
\end{proof}
			\begin{theorem} The $R$-matrix associated with $T_1$ is given by
	\begin{flalign*}
					& R=r^{\frac{1}{2}} s^{-\frac{1}{2}} \sum_{i=1}^{2 n} E_{i i} \otimes E_{i i}+r^{-\frac{1}{2}} s^{\frac{1}{2}} \sum_{i=1}^{2 n} E_{i i} \otimes E_{i^{'} i^{'}}+r^{\frac{1}{2}} s^{\frac{1}{2}}\left\{\sum_{\substack{1 \leq i \leq n-1 \\
							i+1 \leq j \leq n}} E_{i i} \otimes E_{j j}\right. \\
					& +\sum_{\substack{1 \leq i \leq n-1 \\
							i^{'}+1 \leq j \leq 2 n}} E_{i i} \otimes E_{j j}+\sum_{j=n+2}^{2 n} E_{n n} \otimes E_{j j}+\sum_{\substack{n+1 \leq j \leq 2 n-1 \\
							j+1 \leq i \leq 2 n}} E_{i i} \otimes E_{j j} \\
					& \left.+\sum_{\substack{1 \leq i \leq n-1 \\
							n+1 \leq j \leq 2 n-i}} E_{j j} \otimes E_{i i}\right\}+r^{-\frac{1}{2}} s^{-\frac{1}{2}}\left\{\sum_{\substack{1 \leq i \leq n-1 \\
							i+1 \leq j \leq n}} E_{j j} \otimes E_{i i}+\sum_{\substack{1 \leq i \leq n-1 \\
							i^{'}+1 \leq j \leq 2 n}} E_{j j} \otimes E_{i i}\right.
	\end{flalign*}
	\begin{flalign*}
	& \left.+\sum_{j=n+2}^{2 n} E_{j j} \otimes E_{n n}+\sum_{\substack{n+1 \leq j \leq 2 n-1 \\
			j+1 \leq i \leq 2 n}} E_{j j} \otimes E_{i i}+\sum_{\substack{1 \leq i \leq n-1 \\
			n+1 \leq j \leq 2 n-i}} E_{i i} \otimes E_{j j}\right\} \\
	& +\left(r^{\frac{1}{2}} s^{-\frac{1}{2}}-r^{-\frac{1}{2}} s^{\frac{1}{2}}\right)\left\{\sum_{i<j} E_{i j} \otimes E_{j i}-\sum_{i>j}\left(r^{\frac{1}{2}} s^{-\frac{1}{2}}\right)^{\bar{\imath}-\bar{\jmath}} \varepsilon_i \varepsilon_j E_{i^{\prime} j^{\prime}} \otimes E_{i j}\right\},
	\end{flalign*}
where
$$
\varepsilon_i=\left\{\begin{aligned}
	1, & \quad\text { for } \quad i=1, \ldots, n, \\
	-1, & \quad\text { for } \quad i=n+1, \ldots, 2 n,
\end{aligned}\right.
$$
and $(\overline{1}, \overline{2}, \ldots, \overline{2 n})=(n, n-1, \ldots, 1,-1, \ldots,-n)$.
\end{theorem}
\begin{proof}
	For $i=1,2, \cdots ,n-1$, we have $\Delta\left(e_i\right)=e_i \otimes \omega_i+1 \otimes e_i$ with the matrix form given by
\begin{equation}\label{e:Delta-ei}
\begin{aligned}
	&r^{-\frac{3}{2}} s^{-\frac{1}{2}} E_{i, i} \otimes E_{i+1, i}+r^{-\frac{1}{2}} s^{-\frac{3}{2}} E_{i+1, i+1} \otimes E_{i+1, i}+r^{-\frac{1}{2}} s^{\frac{1}{2}} E_{2 n-i, 2 n-i} \otimes E_{i+1, i}\\
	&+r^{\frac{1}{2}} s^{-\frac{1}{2}} E_{2 n-i+1,2 n-i+1} \otimes E_{i+1, i}+r^{-\frac{1}{2}} s^{-\frac{1}{2}} \sum_{j \neq\{i, i+1,2 n-i, 2 n-i+1\}} E_{j, j} \otimes E_{i+1, i}\\
	&-r^{-1} E_{i, i} \otimes E_{2 n-i+1,2 n-i}-s^{-1} E_{i+1, i+1} \otimes E_{2 n-i+1,2 n-i}-s E_{2 n-i, 2 n-i} \otimes E_{2 n-i+1,2 n-i}\\
	&-r E_{2 n-i+1,2 n-i+1} \otimes E_{2 n-i+1,2 n-i}-\sum_{j \neq\{i, i+1,2 n-i, 2 n-i+1\}} E_{j, j} \otimes E_{2 n-i+1,2 n-i}\\
	&+r^{-\frac{1}{2}} s^{-\frac{1}{2}} E_{i+1, i} \otimes\left(\sum_{j=1}^{2 n} E_{j, j}\right)-E_{2 n-i+1,2 n-i} \otimes\left(\sum_{j=1}^{2 n} E_{j, j}\right).
\end{aligned}
\end{equation}	
and $\Delta^{c o p}\left(e_i\right)=\omega_i \otimes e_i+e_i \otimes 1$, whose matrix is the flip of \eqref{e:Delta-ei}.
Here the flip means $A\otimes B\mapsto B\otimes A$ term by term.
Then we can check that
$\Delta^{c o p}\left(e_i\right) R=R \Delta\left(e_i\right)$, for $1\leq i\leq n-1$. Here we only consider nontrivial equations.

To see this, note that for $i=1,2, \cdots ,n-1$, one has that
$$
\left(r^{-\frac{1}{2}} s^{-\frac{3}{2}} E_{i+1, i} \otimes E_{i+1, i+1}\right)\left(r^{\frac{1}{2}} s^{-\frac{1}{2}}-r^{-\frac{1}{2}} s^{\frac{1}{2}}\right)\left(E_{i, i+1} \otimes E_{i+1, i}\right)
$$
$$
+\left(r^{-\frac{1}{2}} s^{-\frac{1}{2}} E_{i+1, i+1} \otimes E_{i+1, i}\right)\left(r^{-\frac{1}{2}} s^{-\frac{1}{2}} E_{i+1, i+1} \otimes E_{i, i}\right)
$$
$$
=\left(r^{\frac{1}{2}} s^{-\frac{1}{2}} E_{i+1, i+1} \otimes E_{i+1, i+1}\right)\left(r^{-\frac{1}{2}} s^{-\frac{3}{2}} E_{i+1, i+1} \otimes E_{i+1, i}\right),
$$

Also we compute that
$$
\left(r^{\frac{1}{2}} s^{-\frac{1}{2}} E_{i+1, i} \otimes E_{2 n-i+1,2 n-i+1}\right)\left(r^{\frac{1}{2}} s^{-\frac{1}{2}}-r^{-\frac{1}{2}} s^{\frac{1}{2}}\right)\left[\left(1+\left(r^{\frac{1}{2}} s^{-\frac{1}{2}}\right)^{\rho_1}\right) E_{i, 2 n-i+1} \otimes E_{2 n-i+1, i}\right]
$$
$$
+\left(-E_{i+1, i+1} \otimes E_{2 n-i+1,2 n-i}\right)\left(r^{\frac{1}{2}} s^{-\frac{1}{2}}-r^{-\frac{1}{2}} s^{\frac{1}{2}}\right)\left[\left(r^{\frac{1}{2}} s^{-\frac{1}{2}}\right)^{\rho_2} E_{i+1,2 n-i+1} \otimes E_{2 n-i, i}\right]
$$
$$
=\left[\left(r^{\frac{1}{2}} s^{-\frac{1}{2}}-r^{-\frac{1}{2}} s^{\frac{1}{2}}\right) E_{i+1,2 n-i+1} \otimes E_{2 n-i+1, i+1}\right]\left(r^{\frac{1}{2}} s^{-\frac{1}{2}} E_{2 n-i+1,2 n-i+1} \otimes E_{i+1, i}\right),
$$
where $
\rho_1=\overline{2 n-i+1}-\bar{i}$ and $\rho_2=\overline{2 n-i}-\bar{i}.$\\

Next, for $i=1,2, \cdots ,n-1, j \neq 2 n-i, 2 n-i+1$, and $j>i+1,$
$$
\left(r^{-\frac{1}{2}} s^{-\frac{1}{2}} E_{i+1, i} \otimes E_{j, j}\right)\left[\left(r^{\frac{1}{2}} s^{-\frac{1}{2}}-r^{-\frac{1}{2}} s^{\frac{1}{2}}\right) E_{i, j} \otimes E_{j, i}\right]
$$
$$
=\left[\left(r^{\frac{1}{2}} s^{-\frac{1}{2}}-r^{-\frac{1}{2}} s^{\frac{1}{2}}\right) E_{i+1, j} \otimes E_{j, i+1}\right]\left(r^{-\frac{1}{2}} s^{-\frac{1}{2}} E_{j, j} \otimes E_{i+1, i}\right),
$$
and for $i=1,2, \cdots ,n-2$,
$$
\left(-E_{i+1, i+1} \otimes E_{2 n-i+1,2 n-i}\right)\left(r^{-\frac{1}{2}} s^{\frac{1}{2}} E_{i+1, i+1} \otimes E_{2 n-i, 2 n-i}\right)
$$
$$
+\left(r^{\frac{1}{2}} s^{-\frac{1}{2}} E_{i+1, i} \otimes E_{2 n-i+1,2 n-i+1}\right)\left[-\left(r^{\frac{1}{2}} s^{-\frac{1}{2}}-r^{-\frac{1}{2}} s^{\frac{1}{2}}\right)\left(r^{\frac{1}{2}} s^{-\frac{1}{2}}\right)^{\rho_3} E_{i, i+1} \otimes E_{2 n-i+1,2 n-i}\right]
$$
$$
=\left(r^{\frac{1}{2}} s^{\frac{1}{2}} E_{i+1, i+1} \otimes E_{2 n-i+1,2 n-i+1}\right)\left(-s^{-1} E_{i+1, i+1} \otimes E_{2 n-i+1,2 n-i}\right),
$$
where $\rho_3=\overline{2 n-i+1}-\overline{2 n-i}.$\\

 Finally, we can check that for $i=1,2, \cdots ,n-1,$
$$
\left(-E_{i+1, i+1} \otimes E_{2 n-i+1,2 n-i}\right)\left(r^{\frac{1}{2}} s^{-\frac{1}{2}}-r^{-\frac{1}{2}} s^{\frac{1}{2}}\right)\left[\left(1+\left(r^{\frac{1}{2}} s^{-\frac{1}{2}}\right)^{\rho_4}\right) E_{i+1,2 n-i} \otimes E_{2 n-i, i+1}\right]
$$
$$
+\left(r^{\frac{1}{2}} s^{-\frac{1}{2}} E_{i+1, i} \otimes E_{2 n-i+1,2 n-i+1}\right)\left(r^{\frac{1}{2}} s^{-\frac{1}{2}}-r^{-\frac{1}{2}} s^{\frac{1}{2}}\right)\left[\left(r^{\frac{1}{2}} s^{-\frac{1}{2}}\right)^{\rho_5} E_{i, 2 n-i} \otimes E_{2 n-i+1, i+1}\right]
$$
$$
=\left[\left(r^{\frac{1}{2}} s^{-\frac{1}{2}}-r^{-\frac{1}{2}} s^{\frac{1}{2}}\right) E_{i+1,2 n-i+1} \otimes E_{2 n-i+1, i+1}\right]\left(-E_{2 n-i+1,2 n-i} \otimes E_{i+1, i+1}\right),
$$
and also
$$
\left(-s^{-1} E_{2 n-i+1,2 n-i} \otimes E_{i+1, i+1}\right)\left(r^{-\frac{1}{2}} s^{\frac{1}{2}} E_{2 n-i, 2 n-i} \otimes E_{i+1, i+1}\right)
$$
$$
=\left(r^{-\frac{1}{2}} s^{-\frac{1}{2}} E_{2 n-i+1,2 n-i+1} \otimes E_{i+1, i+1}\right)\left(-E_{2 n-i+1,2 n-i} \otimes E_{i+1, i+1}\right).
$$
These relations  imply that the $R$-matrix does satisfy $\Delta^{c o p}\left(e_i\right) R=R \Delta\left(e_i\right)$,for $1\leq i\leq n-1$.
Similarly we can check that $
\Delta^{c o p}\left(e_n\right) R=R \Delta\left(e_n\right)$ holds.  Note that the corresponding equations for $f_i$ immediately follow
by $\tau$ and the flip operation on tensor product. Finally, the equations $\Delta^{c o p}\left(w_i\right) R=R \Delta\left(w_i\right)$ and $\Delta^{c o p}\left(w_i^{\prime}\right) R=R \Delta\left(w_i^{\prime}\right)$ are trivial, thus we have shown that the given $R$-matrix is correct.
\end{proof}

	\section{FRT-REALIZATION OF $U_{r, s}\left(\mathrm{C}_n\right)$}
	In this section, we give the Faddeev-Reshetikhin-Takhtajan
	and Drinfeld-Jimbo presentations of $U_{r, s}\left(\mathfrak{sp}_{2n}\right)$, which is more involved than the type $A$ case \cite{JL1}. First, let's give some basic definitions and notations.
	 	\begin{defi}
	 		$U(R)$ is an unital associative algebra with generators $\ell_{i j}^{+}, \ell_{j i}^{-}$ $1 \leq i \leq 2 n$ subject to the following relations. Let $L^{ \pm}=\left(\ell_{i j}^{ \pm}\right), 1 \leq i, j \leq 2 n$, with
\begin{align*}
&\ell_{i j}^{+}=\ell_{j i}^{-}=0,\\
&\ell^{\pm}_{ii}(\ell^{\pm}_{ii})^{-1}=(\ell^{\pm}_{ii})^{-1}\ell^{\pm}_{ii}=1,
\end{align*}
for $1 \leq j<i \leq 2 n$. The defining relations are given in matrix form:
	 		$$
	 		R L_1^{ \pm} L_2^{ \pm}=L_2^{ \pm} L_1^{ \pm} R, \quad R L_1^{+} L_2^{-}=L_2^{-} L_1^{+}R,
	 		$$
	 		where $L_1^{ \pm}=L^{ \pm} \otimes 1, L_2^{ \pm}=1 \otimes L^{ \pm}$.
	 		\end{defi}
Since the diagonal elements of $L^{\pm}$ are invertible, the inverse $\left(L^{ \pm}\right)^{-1}$ exists and its matrix elements belong to
$U(R)$. It follows from the relations between $L_1^{ \pm}$ and $L_2^{ \pm}$ that the following holds.
	 	\begin{proposition}
	 		Let $L^{\prime \pm}=\left(L^{ \pm}\right)^{-1}, L^{\prime \prime \pm}=\left(\left(L^{\prime \pm}\right)^t\right)^{-1}.$ Then
	 		$$
	 		\begin{aligned}
	 			R_{21} L_1^{\prime \pm} L_2^{\prime \pm} & =L_2^{\prime \pm} L_1^{\prime \pm} R_{21}, \quad R_{21} L_1^{\prime+} L_2^{\prime-}=L_2^{\prime-} L_1^{\prime+} R_{21}, \\
	 			R L_1^{\prime \prime \pm} L_2^{\prime \prime \pm} & =L_2^{\prime \prime \pm} L_1^{\prime \prime} R, \quad R L_1^{\prime \prime+} L_2^{\prime \prime-}=L_2^{\prime \prime-} L_1^{\prime \prime+} R.
	 		\end{aligned}
	 		$$
	 \end{proposition}
 \begin{defi}
 	$U(R)$ is a Hopf algebra with comultiplication $\Delta$ defined by
 	$$\Delta\left(L^{ \pm}\right)=L^{ \pm} \dot{\otimes} L^{ \pm},$$
 	whose entry form is given by
 	$$
 	\Delta\left(\ell_{i j}^{ \pm}\right)=\sum_{k=1}^n \ell_{i k}^{ \pm} \otimes \ell_{k j}^{ \pm}.$$ The antipode $S$ is defined by
 	$$
 	S\left(L^{ \pm}\right)=\left(L^{ \pm}\right)^{-1}.
 	$$
 	The counit is defined by
 	$$
 	\varepsilon(L)=I .
 	$$	
 \end{defi}	
	 	\begin{theorem}\label{T:C3}
	 	For $n=3$,	there is an isomorohism between $U(R)$ and $U_{r, s}\left(\mathrm{C}_3\right)$ given explicitly as follows.
\begin{align*}
&\ell_{11}^{+}\mapsto w_1^{-1} w_2^{-1}\left(w_3^{{\frac{1}{2}}}\right)^{-1}, \quad \ell_{22}^{+}\mapsto  w_2^{-1}\left(w_3^{{\frac{1}{2}}}\right)^{-1}, \quad \ell_{33}^{+}\mapsto\left(w_3^{{\frac{1}{2}}}\right)^{-1}, \\
&\ell_{44}^{+}\mapsto w_3^{\frac{1}{2}}, \quad \ell_{55}^{+}\mapsto w_2 w_3^{\frac{1}{2}}, \quad \ell_{66}^{+}\mapsto w_1 w_2 w_3^{\frac{1}{2}},\\
&\ell_{11}^{-}\mapsto\left(w_1^{\prime}\right)^{-1}\left(w_2^{\prime}\right)^{-1}\left(w_3^{\prime}\right)^{-\frac{1}{2}},\quad \ell_{22}^{-}\mapsto\left(w_2^{\prime}\right)^{-1}\left(w_3^{\prime}\right)^{-\frac{1}{2}}, \quad \ell_{33}^{-}\mapsto\left(w_3^{\prime}\right)^{-\frac{1}{2}},\\
&\ell_{44}^{-}\mapsto\left(w_3'\right)^{\frac{1}{2}}, \quad \ell_{55}^{-}\mapsto w_2^{\prime}\left(w_3'\right)^{\frac{1}{2}}, \quad \ell_{66}^{-}\mapsto w_1^{\prime} w_2^{\prime}\left(w_3'\right)^{\frac{1}{2}},\\
&\ell_{12}^{+}\mapsto(r-s) s^{-1} w_1^{-1} w_2^{-1}\left(w_3^{\frac{1}{2}}\right)^{-1} e_1, \quad \ell_{21}^{-}\mapsto-(r-s) r^{-1} f_1\left(w_1^{\prime}\right)^{-1}\left(w_2^{\prime}\right)^{-1}\left(w_3^{\prime}\right)^{-\frac{1}{2}},\\
&\ell_{23}^{+}\mapsto(r-s) s^{-1} w_2^{-1}\left(w_3^{\frac{1}{2}}\right)^{-1} e_2, \quad \ell_{32}^{-}\mapsto-(r-s) r^{-1} f_2\left(w_2^{\prime}\right)^{-1}\left(w_3^{\prime}\right)^{-\frac{1}{2}},\\
&\ell_{34}^{+}\mapsto\left(r^2-s^2\right) s^{-1}\left(w_3^{\frac{1}{2}}\right)^{-1} e_3, \quad \ell_{43}^{-}\mapsto-\left(r^2-s^2\right) s^{-1} f_3\left(w_3^{\prime}\right)^{-\frac{1}{2}},\\
&\ell_{45}^{+}\mapsto-(r-s) w_3^{\frac{1}{2}} e_2, \quad \ell_{54}^{-}\mapsto(r-s) f_2\left(w_3^{\prime}\right)^{\frac{1}{2}},\\
&\ell_{56}^{+}\mapsto-(r-s) w_2 w_3^{\frac{1}{2}} e_1, \quad \ell_{65}^{-}\mapsto(r-s) f_1 w_2^{\prime}\left(w_3^{\prime}\right)^{\frac{1}{2}}.
\end{align*}
\begin{proof} We  remark that the anti-automorphism $\ell^+_{ij}\mapsto \ell^-_{ji}$ corresponds to $e_i\mapsto f_i, \omega_i\mapsto \omega_i', r\mapsto s$.
So the relations can be grouped into two sets, and we only need to verify one set.
Recall that $L^{+}=(\ell^+_{ij})_{6\times 6}$ (resp. $L^{-}=(\ell^-_{ij})_{6\times 6}$) is an upper (resp. lower) triangular matrix, then
	 		$$
	 		L_1^{+}=\left(\begin{array}{cccc}
	 			\ell_{11}^{+} I_6 & \ell_{12}^{+} I_6 & \cdots & \ell_{16}^{+} I_6 \\
	 			0 & \ell_{22}^{+} I_6 & \ddots & \vdots \\
	 			\vdots & \ddots & \ddots & \ell_{56}^{+} I_6 \\
	 			0 & \cdots & 0 & \ell_{66}^{+} I_6
	 		\end{array}\right)_{36 \times 36} \quad, \quad L_1^{-}=\left(\begin{array}{cccc}
	 			\ell_{11}^{-} I_6 & 0 & \cdots & 0 \\
	 			\ell_{21}^{-} I_6 & l_{22}^{-} I_6 & \ddots & \vdots \\
	 			\vdots & \ddots & \ddots & 0 \\
	 			\ell_{61}^{-} I_6 & \cdots & \ell_{65}^{-} I_6 & \ell_{66}^{-} I_6
	 		\end{array}\right)_{36 \times 36}.
	 		$$
Note that $L_2^{\pm}=diag(L^{\pm}, \ldots, L^{\pm})_{36}$.
Applying $R L_1^{+} L_2^{-}=L_2^{-} L_1^{+}R$ to the vector $(v_i\otimes v_j)$,  we get $36$ sets of relations among $\ell^{\mp}_{ij}$.   We list only some typical relations.
For example, $R L_1^{+} L_2^{-}(v_1 \otimes v_j)=L_2^{-} L_1^{+} R(v_1 \otimes v_j)$ leads to the following relations:
\begin{align*}
&\ell_{11}^{+} \ell_{ii}^{-}=\ell_{ii}^{-} \ell_{11}^{+}, \qquad i=1, \ldots, 6,\\
&\ell_{11}^{+} \ell_{32}^{-}=\ell_{32}^{-} \ell_{11}^{+}, \quad \ell_{11}^{+} \ell_{43}^{-}=r s \ell_{43}^{-} \ell_{11}^{+},\\
&\ell_{11}^{+} \ell_{54}^{-}=\ell_{54}^{-} \ell_{11}^{+}, \quad \ell_{11}^{+} \ell_{65}^{-}=s^{-1} \ell_{65}^{-} \ell_{11}^{+}.
\end{align*}
The relation $R L_1^{+} L_2^{-}(v_2 \otimes v_j)=L_2^{-} L_1^{+} R(v_2 \otimes v_j)$ implies  the following commutation relations.
\begin{align*}
&\ell_{22}^{+} \ell_{i i}^{-}=\ell_{i i}^{-} \ell_{22}^{+}, \quad i=1, \ldots, 6,\\
&\ell_{12}^{+} \ell_{43}^{-}=r^{-1} s^{-1} \ell_{43}^{-} \ell_{12}^{+}, \quad \ell_{22}^{+} \ell_{43}^{-}=r^{-1} s^{-1} \ell_{43}^{-} \ell_{22}^{+},\\
&\ell_{22}^{+} \ell_{21}^{-}=r^{-1} \ell_{21}^{-} \ell_{22}^{+}, \quad \ell_{22}^{+} \ell_{65}^{-}=r^{-1} \ell_{65}^{-} \ell_{22}^{+},\\
&\ell_{12}^{+} \ell_{22}^{-}=s^{-1} \ell_{22}^{-} \ell_{12}^{+}, \quad \ell_{22}^{+} \ell_{32}^{-}=s^{-1} l_{32}^{-} \ell_{22}^{+},\\
&\ell_{12}^{+} \ell_{66}^{-}=r \ell_{66}^{-} \ell_{12}^{+}, \quad \ell_{12}^{+} \ell_{55}^{-}=s \ell_{55}^{-} \ell_{12}^{+}.
\end{align*}

Similarly, the equations $R L_1^{+} L_2^{-}\left(v_i \otimes v_j\right)=L_2^{-} L_1^{+} R\left(v_i \otimes v_j\right), 3 \leq i \leq 6,1 \leq j \leq 6$, and the equations $R L_1^{ \pm} L_2^{ \pm}\left(v_i \otimes v_j\right)=L_2^{ \pm} L_1^{ \pm} R\left(v_i \otimes v_j\right), 1 \leq i \leq 6,1 \leq j \leq 6$
will give all so-called first type relations: those among $\ell^+_{ki}$ and $\ell^-_{lj}$. 

Next, we consider the second type relations involving the elements $\ell^{\pm}_{i,i+1}$ and $\ell^{\pm}_{i+1,i}$ etc.
$$
(r-s) \ell_{22}^{+} \ell_{11}^{-}+r s \ell_{12}^{+} \ell_{21}^{-}=(r-s) \ell_{22}^{-} \ell_{11}^{+}+\ell_{21}^{-} \ell_{12}^{+},
$$
$$
(r-s) \ell_{33}^{+} \ell_{22}^{-}+r s \ell_{23}^{+} \ell_{32}^{-}=(r-s) \ell_{33}^{-} \ell_{22}^{+}+\ell_{32}^{-} \ell_{23}^{+},
$$
$$
(r-s)\left(1+r s^{-1}\right) \ell_{44}^{+} \ell_{33}^{-}+s \ell_{34}^{+} \ell_{43}^{-}=(r-s)\left(1+r s^{-1}\right) \ell_{44}^{-} \ell_{33}^{+}+s \ell_{43}^{-} \ell_{34}^{+},
$$
$$
(r-s) \ell_{55}^{+} \ell_{44}^{-}+\ell_{45}^{+} \ell_{54}^{-}=(r-s) \ell_{55}^{-} \ell_{44}^{+}+r s \ell_{54}^{-} \ell_{45}^{+},
$$
$$
(r-s) \ell_{66}^{+} \ell_{55}^{-}+\ell_{56}^{+} \ell_{65}^{-}=(r-s) \ell_{66}^{-} \ell_{55}^{+}+r s \ell_{65}^{-} \ell_{56}^{+} ,
$$
$$
(r-s) \ell_{i i}^{+} \ell_{j j}^{-}+s \ell_{i-1, i}^{+} \ell_{j+1, j}^{-}=(r-s) \ell_{j+1, j+1}^{-} \ell_{i-1, i-1}^{+}+s \ell_{j+1, j}^{-} \ell_{i-1, i}^{+}, \quad\text {for } i+j=7
$$

We now show the Serre relations. It follows from the defining relations that
 $$
 (r-s) \ell_{22}^{+} \ell_{13}^{+}+r s \ell_{12}^{+} \ell_{23}^{+}=r s \ell_{23}^{+} \ell_{12}^{+} \text {. }
 $$
 Multiplying $\ell_{12}^{+}$ from the left and the right, we have that
 $$
 \begin{aligned}
 	& r s\left(\ell_{12}^{+}\right)^2 \ell_{23}^{+}+(r-s) \ell_{12}^{+} \ell_{22}^{+} \ell_{13}^{+}=r s \ell_{12}^{+} \ell_{23}^{+} \ell_{12}^{+}, \\
 	& r s \ell_{12}^{+} \ell_{23}^{+} \ell_{12}^{+}+(r-s) \ell_{22}^{+} \ell_{13}^{+} \ell_{12}^{+}=r s \ell_{23}^{+}\left(\ell_{12}^{+}\right)^2 ,
 \end{aligned}
 $$
 Using the relations
 $
 \ell_{13}^{+} \ell_{12}^{+}=s^{-1} \ell_{12}^{+} \ell_{13}^{+}, \ell_{13}^{+} \ell_{22}^{+}=r^{-1} s^{-1} \ell_{22}^{+} \ell_{13}^{+} ,
 $
 we get that
 $$
 \left(\ell_{12}^{+}\right)^2 \ell_{23}^{+}-(r+s) \ell_{12}^{+} \ell_{23}^{+} \ell_{12}^{+}+r s \ell_{23}^{+}\left(\ell_{12}^{+}\right)^2=0.
 $$

 Similarly, first using defining relation we get that
 $$
 (r-s) \ell_{33}^{+} \ell_{24}^{+}+r s \ell_{23}^{+} \ell_{34}^{+}=-(r-s) r^{-\frac{1}{2}} s^{\frac{1}{2}} \ell_{35}^{+} \ell_{22}^{+}+s \ell_{34}^{+} \ell_{23}^{+},
 $$
 also
 $$
 \ell_{22}^{+} \ell_{23}^{+}=s \ell_{23}^{+} \ell_{22}^{+} \text {, }
 $$
 $$
 r \ell_{24}^{+} \ell_{23}^{+}=s^{-1}\left(r^2-s^2\right) \ell_{24}^{+} \ell_{23}^{+}+s \ell_{23}^{+} \ell_{24}^{+},
 $$
 $$
 \ell_{35}^{+} \ell_{23}^{+}=(r-s) \ell_{35}^{+} \ell_{23}^{+}+\ell_{23}^{+} \ell_{35}^{+} \text {, }
 $$
 It follows immediately that 
 $$
 \left(\ell_{23}^{+}\right)^3 \ell_{34}^{+}-\left(r+(r s)^{\frac{1}{2}}+s\right)\left(\ell_{23}^{+}\right)^2 \ell_{34}^{+} \ell_{23}^{+}+(r s)^{\frac{1}{2}}\left(r+(r s)^{\frac{1}{2}}+s\right) \ell_{23}^{+} \ell_{34}^{+}\left(\ell_{23}^{+}\right)^2-(r s)^{\frac{3}{2}} \ell_{34}^{+}\left(\ell_{23}^{+}\right)^3=0.
 $$
\end{proof}
\end{theorem}		

 \begin{theorem}
 	There is an isomorohism between $U(R)$ and $U_{r, s}\left(\mathrm{C}_n\right)$ given as follows. For $1\leq i\leq n$
\begin{align*}
 	\ell_{i i}^{+}& \mapsto \omega_i^{-1} \omega_{i+1}^{-1} \cdots \omega_{n-1}^{-1}\left(\omega_n^{\frac{1}{2}}\right)^{-1}, \\
 	\ell_{i^{'} i^{'}}^{+}& \mapsto \omega_i \omega_{i+1} \cdots \omega_{n-1} \omega_n^{\frac{1}{2}},\\
 	\ell_{i i}^{-}& \mapsto\left(\omega_i^{\prime}\right)^{-1}\left(\omega_{i+1}^{\prime}\right)^{-1} \cdots\left(\omega_{n-1}^{\prime}\right)^{-1}\left(\omega_n^{\prime}\right)^{-\frac{1}{2}},\\
 	\ell_{i^{'} i^{'}}^{-}& \mapsto\omega_i^{\prime} \omega_{i+1}^{\prime} \cdots \omega_{n-1}^{\prime}\left(\omega_n^{\prime}\right)^{\frac{1}{2}},
 \end{align*}
 \begin{align*}
 	\ell_{k, k+1}^{+}&\mapsto (r-s) s^{-1} \ell_{k k}^{+} e_k, \quad \ell_{k+1, k}^{-}\mapsto -(r-s) r^{-1} f_k \ell_{k k}^{-},\\
\ell_{2 n-k, 2 n-k+1}^{+}&\mapsto -(r-s) \ell_{2 n-k, 2 n-k}^{+} e_k,\quad \ell_{2 n-k+1,2 n-k}^{-}\mapsto (r-s) f_k \ell_{2 n-k, 2 n-k}^{-} \text {, }\\
 	\ell_{n, n+1}^{+}& \mapsto \left(r^2-s^2\right) s^{-1} \ell_{n n}^{+} e_n, \quad \ell_{n+1, n}^{-}\mapsto -\left(r^2-s^2\right)s^{-1} f_n l_{n n}^{-},
 \end{align*}
 	where $i^{\prime}=2 n+1-i.$
 	\end{theorem}
 \begin{proof} The isomorphism is proved by induction and the first step was Theorem \ref{T:C3}, and the induction starts from index $n$ and ends with index $1$. In the following remark we show how the induction goes for the first case.
 	\end{proof}
  \begin{remark} One can write down explicit expression for composite root vectors $\ell_{i j}^{+}$ in terms of simple root vectors.
  For example, when n=3, we can get  following ecurrence relations for positive roots.
\begin{align*}
\ell_{46}^{+}&=\frac{1}{r-s}\left(\ell_{55}^{+}\right)^{-1}\left(\ell_{56}^{+} \ell_{45}^{+}-\ell_{45}^{+} \ell_{56}^{+}\right), \quad \ell_{35}^{+}=\frac{1}{r^2-s^2}\left(\ell_{44}^{+}\right)^{-1}\left(s \ell_{45}^{+} \ell_{34}^{+}-s^2 \ell_{34}^{+} \ell_{45}^{+}\right),\\
\ell_{24}^{+}&=\frac{1}{r-s}\left(\ell_{33}^{+}\right)^{-1}\left(-(r-s) r^{-\frac{1}{2}} s^{\frac{1}{2}} \ell_{35}^{+} \ell_{22}^{+}+s \ell_{34}^{+} \ell_{23}^{+}-r s \ell_{23}^{+} \ell_{34}^{+}\right), \\
\ell_{13}^{+}&=\frac{r s}{r-s}\left(\ell_{22}^{+}\right)^{-1}\left(\ell_{23}^{+} \ell_{12}^{+}-\ell_{12}^{+} \ell_{23}^{+}\right), \quad \ell_{36}^{+}=\frac{1}{r-s}\left(\ell_{55}^{+}\right)^{-1}\left(\ell_{56}^{+} \ell_{35}^{+}-r s \ell_{35}^{+} \ell_{56}^{+}\right),\\
\ell_{25}^{+}&=\frac{r s}{r-s}\left(\ell_{33}^{+}\right)^{-1}\left(\ell_{35}^{+} \ell_{23}^{+}-\ell_{23}^{+} \ell_{35}^{+}\right), \quad \ell_{14}^{+}=\frac{1}{r-s}\left(\ell_{22}^{+}\right)^{-1}\left(\ell_{24}^{+} \ell_{12}^{+}-r s \ell_{12}^{+} \ell_{24}^{+}\right),\\
\ell_{26}^{+}&=\frac{1}{r-s}\left(\ell_{44}^{+}\right)^{-1}\left(\ell_{46}^{+} \ell_{24}^{+}-\ell_{24}^{+} \ell_{46}^{+}\right), \quad \ell_{15}^{+}=\frac{r s}{r-s}\left(\ell_{33}^{+}\right)^{-1}\left(\ell_{35}^{+} \ell_{13}^{+}-\ell_{13}^{+} \ell_{35}^{+}\right),\\
\ell_{16}^{+}&=\frac{1}{r-s}\left(\ell_{44}^{+}\right)^{-1}\left(\ell_{46}^{+} \ell_{14}^{+}-\ell_{14}^{+} \ell_{46}^{+}\right).
\end{align*}
Using these and the formula for the negative root vectors, we can obtain Gauss decomposition for $\ell_{i j}^{+}$.

  	\end{remark}

	\section{$RLL$-REALIZATION OF $U_{r, s}(\widehat{\mathfrak{sp}}_{2n})$}
	\subsection{The algebras $U(\bar{R})$}
The matrix $R$ has a spectral parameter dependent $R$-matrix $\bar{R}(z)$ obtained by the Yang-Baxterization method (cf. \cite{GWX}).
Since $R$ has three different eigenvalues, the spectral dependent $R$-matrix $\bar{R}(u)$ is given by
$$
\bar{R}(u)=r^{\frac{1}{2}} s^{\frac{1}{2}} \frac{u-1}{u r-s} R+\frac{r-s}{u r-s} P-\frac{(r-s)(u-1) r^{-n-1} s^{n+1}}{(u r-s)\left(u-r^{-n-1} s^{n+1}\right)} Q.
$$
where $P=\sum_{ij}E_{ij}\otimes E_{ji}$ is the permutation operator and $Q=\sum_{ij}(r^{\frac{1}{2}}s^{-\frac{1}{2}})^{\bar{i}-\bar{j}}\varepsilon_{i}\varepsilon_{j}E_{i'j'}\otimes E_{ij}$.
This method depends on explicit information of the eigenvalues of $R$. There is a
representation theoretic method to compute the quantum affine $R$ matrix given by Jimbo \cite{J}, which we now generalize to the
two-parameter case.

Recall that the universal $R$-matrix $\mathcal{R}\in \mathcal{A} \hat{\otimes} \mathcal{A}$ of a quasitriangular Hopf algebra $\mathcal{A}$ satisfies the following equation
\begin{equation}\label{e:univR}
\Delta^{o p}(x) \mathcal{R}=\mathcal{R} \Delta(x),
\end{equation}
for any $x \in \mathcal{A}$. Here $\mathcal{A} \hat{\otimes} \mathcal{A}$ refers to some completed tensor product. The relation \eqref{e:univR} is
equivalent to the following equations: on any representation $\pi$ of $\mathcal{A}$, let $R\left(\frac{u}{v}\right)=\left(\pi_u \otimes \pi_v\right) \mathcal{R}$ then
\begin{align}\label{e:univR1}
&\left[R\left(\frac{u}{v}\right), \omega_i \otimes \omega_i\right]=\left[R\left(\frac{u}{v}\right), \omega_i^{\prime} \otimes \omega_i^{\prime}\right]=0,\\
\label{e:univR2}
& R\left(\frac{u}{v}\right)\left(e_i \otimes 1+\omega_i \otimes e_i\right)=\left(1 \otimes e_i+e_i \otimes \omega_i\right) R\left(\frac{u}{v}\right),\\ \label{e:univR3}
&R\left(\frac{u}{v}\right)\left(f_i \otimes \omega_i^{\prime}+1 \otimes f_i\right)=\left(\omega_i^{\prime} \otimes f_i+f_i \otimes 1\right) R\left(\frac{u}{v}\right),
\end{align}
for $i=0,1, \cdots ,n$. It is known \cite{J} that the dimension of the solution space of the linear system \eqref{e:univR1}-\eqref{e:univR3} is at most $1$.
Therefore the $R$-matrix $\bar{R}(z)$ is unique up to scale function.
	
We now use this method to determine the spectral dependent $R$-matrix. The fundamental representation $T_1$ of the finite-dimensional $U_{r, s}(\mathfrak{s p}_{2 n})$ can be lifited to that of $U_{r, s}(\widehat{\mathfrak{s p}}_{2 n})$ by definiting
\begin{align*}
T_1\left(\omega_0\right)&=r s^{-1} E_{1,1}+r s E_{2,2}+r^{-1} s^{-1} E_{2 n-1,2 n-1}+r^{-1} s E_{2 n, 2 n}+\sum_{j \neq\{1,2,2 n-1,2 n\}} E_{j, j},\\
T_1\left(w_0^{\prime}\right)&=r^{-1} s E_{1,1}+r s E_{2,2}+r^{-1} s^{-1} E_{2 n-1,2 n-1}+r s^{-1} E_{2 n, 2 n}+\sum_{j \neq\{1,2,2 n-1,2 n\}} E_{j, j},\\
T_1\left(e_0\right)&=r^{-\frac{1}{2}} s^{-\frac{1}{2}} u E_{1,2 n}, \quad
T_1\left(f_0\right)=r^{-\frac{1}{2}} s^{-\frac{1}{2}} u^{-1} E_{2 n, 1}.
\end{align*}
 After a detailed and straightforward computation with \eqref{e:univR1}-\eqref{e:univR3} we find the quantum affine $R$ matrix.
 \begin{proposition} The quantum affine $R$-matrix for $U_{r, s}(\widehat{\mathfrak{sp}}_{2n})$ is given by
	\begin{flalign*}
 			& \bar{R}(u)=\sum_{i=1}^{2 n} E_{i i} \otimes E_{i i}+r s \frac{u-1}{u r-s}\left\{\sum_{\substack{1 \leq i \leq n-1 \\
 					i+1 \leq j \leq n}} E_{i i} \otimes E_{j j}\right. \\
 			& +\sum_{\substack{1 \leq i \leq n-1 \\
 					i^{'}+1 \leq j \leq 2 n}} E_{i i} \otimes E_{j j}+\sum_{j=n+2}^{2 n} E_{n n} \otimes E_{j j}+\sum_{\substack{n+1 \leq j \leq 2 n-1 \\
 					j+1 \leq i \leq 2 n}} E_{i i} \otimes E_{j j} \\
 			& \left.+\sum_{\substack{1 \leq i \leq n-1 \\
 					n+1 \leq j \leq 2 n-i}} E_{j j} \otimes E_{i i}\right\}+\frac{u-1}{u r-s}\left\{\sum_{\substack{1 \leq i \leq n-1 \\
 					i+1 \leq j \leq n}} E_{j j} \otimes E_{i i}+\sum_{\substack{1 \leq i \leq n-1 \\
 					i^{'}+1 \leq j \leq 2 n}} E_{j j} \otimes E_{i i}\right.
	\end{flalign*}
	\begin{flalign*}
	& \left.+\sum_{j=n+2}^{2 n} E_{j j} \otimes E_{n n}+\sum_{\substack{n+1 \leq j \leq 2 n-1 \\
			j+1 \leq i \leq 2 n}} E_{j j} \otimes E_{i i}+\sum_{\substack{1 \leq i \leq n-1 \\
			n+1 \leq j \leq 2 n-i}} E_{i i} \otimes E_{j j}\right\} \\
	& +\frac{r-s}{u r-s} \sum_{i>j, i \neq j^{'}} E_{i j} \otimes E_{j i}+\frac{(r-s) u}{u r-s} \sum_{i<j, i \neq j^{\prime}} E_{i j} \otimes E_{j i} \\
	& +\frac{1}{\left(u-r^{-1} s\right)\left(u-\left(r s^{-1}\right)^{-n-1}\right)} \sum_{i . j=1}^{2 n} a_{i j}(u) E_{i^{\prime} j^{\prime}} \otimes E_{i j}, \\
	&
	\end{flalign*}
where
$$
a_{i j}(u)=\left\{\begin{array}{l}
	(u-1)\left(r^{-1} s u-r^{-n-1} s^{n+1}\right), \quad for  $i=j$, \\
	\left(r^{-1} s-1\right) u\left\{\left(r^{\frac{1}{2}} s^{-\frac{1}{2}}\right)^{\bar{\imath}-\bar{\jmath}} \varepsilon_i \varepsilon_j(u-1)-\delta_{i j^{\prime}}\left(u-r^{-n-1} s^{n+1}\right)\right\}  \mbox{,\quad for $i>j$,}\\
	\left(r^{-1} s-1\right)\left\{\left(r^{\frac{1}{2}} s^{-\frac{1}{2}}\right)^{\bar{\imath}-\bar{\jmath}} \varepsilon_i \varepsilon_j r^{-n-1} s^{n+1}(u-1)-\delta_{i j^{\prime}}\left(u-r^{-n-1} s^{n+1}\right)\right\}  \mbox{,\quad for $i<j$}.
\end{array}\right.
$$	
	\end{proposition}
It is easy to check that $\bar{R}(u)$ satisfies the quantum Yang-Baxter equation
$$
\bar{R}_{12}(z) \bar{R}_{13}(z w) \bar{R}_{23}(w)=\bar{R}_{23}(w) \bar{R}_{13}(z w) \bar{R}_{12}(z),
$$
and the unitary condition
$$
\bar{R}_{21}(z) \bar{R}_{12}\left(z^{-1}\right)=\bar{R}_{12}\left(z^{-1}\right) \bar{R}_{21}(z)=1.
$$

We define the algebra $U(\bar{R})$ as the unital associative algebra generated by two invertible central elements $r^{c / 2}$,$s^{c / 2}$ and elements $\ell_{i j}^{ \pm}[\mp m]$ with $1 \leqslant i, j \leqslant 2 n$ and $m \in \mathbb{Z}_{+}$ such that
\begin{align*}
&\ell_{i j}^{+}[0]=\ell_{j i}^{-}[0]=0  ,\text { for }  i>j, \\
&\ell_{n, n}^{+}[0] \ell_{n+1, n+1}^{+}[0]=1, \quad \text { and } \quad \ell_{i i}^{+}[0] \ell_{i i}^{-}[0]=\ell_{i i}^{-}[0] \ell_{i i}^{+}[0]=1.
\end{align*}
and
$$
\bar{R}\left(\frac{z}{w}\right) L_1^{ \pm}(z) L_2^{ \pm}(w)=L_2^{ \pm}(w) L_1^{ \pm}(z) \bar{R}\left(\frac{z}{w}\right),
$$
$$
\bar{R}\left(\frac{z_{+}}{w_{-}}\right) L_1^{+}(z) L_2^{-}(w)=L_2^{-}(w) L_1^{+}(z) \bar{R}\left(\frac{z_{-}}{w_{+}}\right).
$$
where $z_{+}=z r^{\frac{c}{2}}$ and $z_{-}=z s^{\frac{c}{2}}.$

	\subsection{Gauss Decomposition}
We first recall the notion of quasideterminants of Gelfand and Retakh \cite{GR}.
		\begin{defi}
			Let $X=\left(x_{i j}\right)_{i, j=1}^n$ be a matrix over a ring with identity. Denote by $X^{i j}$ the submatrix obtained from $X$ by deleting the ith row and jth
			column.  Suppove the matrix $X^{i j}$ is invertible over the ring. The $(i, j)$-th quasi-determiant $|X|_{i j}$ of $X$ is defined by
			$$
			|X|_{i j}=\left|\begin{array}{rrrrr}
				x_{11} & \cdots & x_{1 j} & \cdots & x_{1 n} \\
				& \cdots & & \cdots & \\
				x_{i 1} & \cdots & \boxed{x_{i j}} & \cdots & x_{i n} \\
				& \cdots & & \cdots & \\
				x_{n 1} & \cdots & x_{n j} & \cdots & x_{n n}
			\end{array}\right|=x_{i j}-r_i^j\left(X^{i j}\right)^{-1} c_j^i,
			$$
			where $r_i^j$ is $i$th row of $X^{ij}$ and $c_j^t$ is $j$th column of $X^{ij}$.
			\end{defi}
Then we have the following Gauss decomposition for the generic matrix $L^{\pm}(z)$ as in \cite{JLM2}.
	 	\begin{proposition}\cite{JLM2}
	 		$L^{ \pm}(z)$ have the following unique decomposition
	 		$$
	 		L^{ \pm}(z)=F^{ \pm}(z) H^{ \pm}(z) E^{ \pm}(z),
	 		$$
	 where the matrices ${H}^{ \pm}(z)=diag(\mathfrak{h}_1^{ \pm}(z), \mathfrak{h}_2^{ \pm}(z), \ldots, \mathfrak{h}_n^{ \pm}(z))$ and		
\begin{align*}
	 		F^{ \pm}(z)=\left(\begin{array}{cccc}
	 			1 & & & 0 \\
	 			\mathfrak{f}_{21}^{ \pm}(z) & \ddots & & \\
	 			\vdots & \ddots & \ddots & \\
	 			\mathfrak{f}_{n 1}^{ \pm}(z) & \cdots & 	\mathfrak{f}_{n, n-1}^{ \pm}(z) & 1
	 		\end{array}\right), \qquad
	 		E^{ \pm}(z)=\left(\begin{array}{cccc}
	 			1 & 	\mathfrak{e}_{12}^{ \pm}(z) & \cdots & 	\mathfrak{e}_{1 n}^{ \pm}(z) \\
	 			& \ddots & \ddots & \vdots \\
	 			& & \ddots & 	\mathfrak{e}_{n-1, n}^{ \pm}(z) \\
	 			0 & & & 1
	 		\end{array}\right),
\end{align*}
	 		where the entries are given by the quasidetermiants:
	 		$$
	 		\mathfrak{h}_m^{ \pm}(z)=\left|\begin{array}{cccc}
	 			\ell_{11}^{ \pm}(z) & \cdots & \ell_{1, m-1}^{ \pm}(z) & \ell_{1 m}^{ \pm}(z) \\
	 			\vdots & & \vdots & \vdots \\
	 			\ell_{m 1}^{ \pm}(z) & \cdots & \ell_{m, m-1}^{ \pm}(z) & \boxed{\ell_{m m}^{ \pm}(z)}
	 		\end{array}\right|
	 		$$
	 		for $1 \leq s \leq 2 n, \mathfrak{h}_m^{ \pm}(z)=\sum_{s \in \mathbb{Z}_{+}} \mathfrak{h}_m^{ \pm}(\mp s) z^{ \pm s}$.
	 		$$
	 		\mathfrak{e}_{i j}^{ \pm}(z)=\mathfrak{h}_i^{ \pm}(z)^{-1}\left|\begin{array}{cccc}
	 			\ell_{11}^{ \pm}(z) & \cdots & \ell_{1, i-1}^{ \pm}(z) & \ell_{1 j}^{ \pm}(z) \\
	 			\vdots & & \vdots & \vdots \\
	 			\ell_{i 1}^{ \pm}(z) & \cdots & \ell_{i, i-1}^{ \pm}(z) & \boxed{\ell_{i j}^{ \pm}(z)}
	 		\end{array}\right|
	 		$$
	 		for $1 \leq i<j \leq 2 n, 	\mathfrak{e}_{i j}^{ \pm}(z)=\sum_{m \in \mathbb{Z}_{+}} 	\mathfrak{e}_{i j}^{ \pm}(\mp m) z^{ \pm m}$.
	 		$$
	 		\mathfrak{f}_{j i}^{ \pm}(z)=\left|\begin{array}{cccc}
	 			\ell_{11}^{ \pm}(z) & \cdots & \ell_{1, i-1}^{ \pm}(z) & \ell_{1 i}^{ \pm}(z) \\
	 			\vdots & & \vdots & \vdots \\
	 			\ell_{j 1}^{ \pm}(z) & \cdots & \ell_{j, i-1}^{ \pm}(z) & \boxed{\ell_{j i}^{ \pm}(z)}
	 		\end{array}\right| 	\mathfrak{h}_i^{ \pm}(z)^{-1}
	 		$$
	 		for $1 \leq i<j \leq 2 n, 	\mathfrak{f}_{j i}^{ \pm}(z)=\sum_{m \in \mathbb{Z}_{+}} 	\mathfrak{f}_{j i}^{ \pm}(\mp m) z^{ \pm m}$.
 		\end{proposition}
	 		
	\subsection{ Homomorphism theorems}
	Now, we can construct a homomorphism from the algebras $U(\bar{R})$  associated with the Lie algebras $\mathfrak{s p}_{2 n-2}$ into $\mathfrak{s p}_{2 n}$ as in the one-parameter case.  Since the rank n will vary, we will indicate the dependence on $n$ by adding a subscript $[n]$ to the $R$-matrices. Consider the algebra $U\left(\bar{R}^{[n-1]}\right)$,  and let the indices of the generators $\ell_{i j}^{ \pm}[\mp m]$ range over the sets $2 \leqslant i, j \leqslant 2^{\prime}$ and $m=0,1, \ldots$, where $i^{\prime}=2 n-i+1$.
\begin{theorem}\label{t:homo}\cite{JLM2}
	The mappings $r^{c / 2} \mapsto r^{c / 2}, s^{c / 2} \mapsto s^{c / 2}$ and
	$$
	\ell_{i j}^{ \pm}(u) \mapsto\left|\begin{array}{cccc}
		\ell_{11}^{ \pm}(u) & \ldots & \ell_{1 m}^{ \pm}(u) & \ell_{1 j}^{ \pm}(u) \\
		\ldots & \ldots & \ldots & \ldots \\
		\ell_{m 1}^{ \pm}(u) & \ldots & \ell_{m m}^{ \pm}(u) & \ell_{m j}^{ \pm}(u) \\
		\ell_{i 1}^{ \pm}(u) & \ldots & \ell_{i m}^{ \pm}(u) & \boxed{\ell_{i j}^{ \pm}(u)}
	\end{array}\right|, \quad m+1 \leqslant i, j \leqslant(m+1)^{\prime}
	$$
	defines a homomorphism $\psi_m: U\left(\bar{R}^{[n-m]}\right) \rightarrow U\left(\bar{R}^{[n]}\right)$.
\end{theorem}
The case of $m=1$ will be essential what follows. Under the hypothesis of Theorem \ref{t:homo}, we have the following useful results, which are
frequently used in computation. 
\begin{coro}\label{c:15}
	For all $1 \leqslant a, b \leqslant m$, and $m+1 \leq i,j \leq n$ or $n+1 \leq i,j \leq(m+1)^{\prime}$, we have
	$$
	\left[\ell_{a b}^{ \pm}(u), \psi_m\left(\ell_{i j}^{ \pm}(v)\right)\right]=0,
	$$
	$$
	\frac{u_{ \pm}-v_{\mp}}{r u_{ \pm}-s v_{\mp}} \ell_{a b}^{ \pm}(u) \psi_m\left(\ell_{i j}^{\mp}(v)\right)=\frac{u_{\mp}-v_{ \pm}}{r u_{\mp}-s v_{ \pm}} \psi_m\left(\ell_{i j}^{\mp}(v)\right) \ell_{a b}^{ \pm}(u).
	$$
	For all $1 \leqslant a, b \leqslant m$, $m+1 \leq i \leq n$ and $n+1 \leq j \leq(m+1)^{\prime}$ we have
	$$
	r s \ell_{a b}^{ \pm}(u) \psi_m\left(\ell_{i j}^{ \pm}(v)\right)=\psi_m\left(\ell_{i j}^{ \pm}(v)\right) \ell_{a b}^{ \pm}(u),
	$$
	$$
	rs\frac{u_{ \pm}-v_{\mp}}{r u_{ \pm}-s v_{\mp}} \ell_{a b}^{ \pm}(u) \psi_m\left(\ell_{i j}^{\mp}(v)\right)=\frac{u_{\mp}-v_{ \pm}}{r u_{\mp}-s v_{ \pm}} \psi_m\left(\ell_{i j}^{\mp}(v)\right) \ell_{a b}^{ \pm}(u).
	$$
		For all $1 \leqslant a, b \leqslant m$, $m+1 \leq j \leq n$ and $n+1 \leq i \leq(m+1)^{\prime}$ we have
			$$
	 \ell_{a b}^{ \pm}(u) \psi_m\left(\ell_{i j}^{ \pm}(v)\right)=	r s\psi_m\left(\ell_{i j}^{ \pm}(v)\right) \ell_{a b}^{ \pm}(u),
		$$
		$$
		\frac{u_{ \pm}-v_{\mp}}{r u_{ \pm}-s v_{\mp}} \ell_{a b}^{ \pm}(u) \psi_m\left(\ell_{i j}^{\mp}(v)\right)=rs\frac{u_{\mp}-v_{ \pm}}{r u_{\mp}-s v_{ \pm}} \psi_m\left(\ell_{i j}^{\mp}(v)\right) \ell_{a b}^{ \pm}(u).
		$$
\end{coro}
\begin{proof}
	It is easy to check by definition relations of $\ell_{i j}^{ \pm}(u)$ and Theorem \ref{t:homo}.
\end{proof}
		\subsection{Images of the generators under the homomorphism $\psi_m$}
		Suppose that $0 \leqslant m<n$. We will use the superscript $[n-m]$ to indicate square submatrices corresponding to rows and columns labeled by $m+1, m+2, \ldots,(m+1)^{\prime}$.  In particular, we set
		$$
		\mathcal{F}^{ \pm[n-m]}(u)=\left[\begin{array}{cccc}
			1 & 0 & \cdots & 0 \\
			\mathfrak{f}_{m+2 m+1}^{ \pm}(u) & 1 & \cdots & 0 \\
			\vdots & \ddots & \ddots & \vdots \\
			\mathfrak{f}_{(m+1)^{\prime} m+1}^{ \pm}(u) & \ldots & \mathfrak{f}_{(m+1)^{\prime}(m+2)^{\prime}}^{ \pm}(u) & 1
		\end{array}\right],
		$$
		$$
		\mathcal{E}^{ \pm[n-m]}(u)=\left[\begin{array}{cccc}
			1 & \mathfrak{e}_{m+1 m+2}^{ \pm}(u) & \ldots & \mathfrak{e}_{m+1(m+1)^{\prime}}^{ \pm}(u) \\
			0 & 1 & \ddots & \vdots \\
			\vdots & \vdots & \ddots & \mathfrak{e}_{(m+2)^{\prime}(m+1)^{\prime}}^{ \pm}(u) \\
			0 & 0 & \ldots & 1
		\end{array}\right],
		$$
		and $\mathcal{H}^{ \pm[n-m]}(u)=\operatorname{diag}\left[\mathfrak{h}_{m+1}^{ \pm}(u), \ldots, \mathfrak{h}_{(m+1)^{\prime}}^{ \pm}(u)\right].$ Furthermore, introduce the products of these matrices by
		$$
		\mathcal{L}^{ \pm[n-m]}(u)=\mathcal{F}^{ \pm[n-m]}(u) \mathcal{H}^{ \pm[n-m]}(u) \mathcal{E}^{ \pm[n-m]}(u).
		$$
		The entries of $\mathcal{L}^{ \pm[n-m]}(u)$ will be denoted by $\ell_{i j}^{ \pm[n-m]}(u).$
	\begin{remark}
		In particular, when $m=n-1$, we can get
		$$
		\mathcal{F}^{ \pm[1]}(u)=\left(\begin{array}{cc}
			1 & 0 \\
			\mathfrak{f}_{n+1 n}^{ \pm}(u) & 1
		\end{array}\right), \quad \mathcal{H}^{ \pm[1]}(u)=\left(\begin{array}{cc}
			\mathfrak{h}_n^{ \pm}(u) & 0 \\
			0 & \mathfrak{h}_{n+1}^{ \pm}(u)
		\end{array}\right), \quad \mathcal{E}^{ \pm[1]}(u)=\left(\begin{array}{cc}
			1 & \mathfrak{e}_{n n+1}^{ \pm}(u) \\
			0 & 1
		\end{array}\right).
		$$
		Therefore,
		$$
		\ell_{n n}^{ \pm[1]}(u)=\mathfrak{h}_n^{ \pm}(u), \quad \ell_{n n+1}^{ \pm[1]}(u)=\mathfrak{h}_n^{ \pm}(u) \mathfrak{e}_{n n+1}^{ \pm}(u), \quad \ell_{n+1 n}^{ \pm[1]}(u)=\mathfrak{f}_{n+1 n}^{ \pm}(u) \mathfrak{h}_n^{ \pm}(u),
		$$
		$$
		\ell_{n+1 n+1}^{ \pm[1]}(u)=\mathfrak{h}_{n+1}^{ \pm}(u)+\mathfrak{f}_{n+1 n}^{ \pm}(u) \mathfrak{h}_n^{ \pm}(u) \mathfrak{e}_{n n+1}^{ \pm}(u).
		$$
		We will use these equations frequently.
	\end{remark}
\begin{proposition}\cite{JLM2}
The series $\ell_{i j}^{ \pm[n-m]}(u)$ coincides with the image of the generator series $\ell_{i j}^{ \pm}(u)$ of the extended quantum affine algebra $U\left(\bar{R}^{[n-m]}\right),$
$$
\ell_{i j}^{ \pm[n-m]}(u)=\psi_m\left(\ell_{i j}^{ \pm}(u)\right), \quad m+1 \leqslant i, j \leqslant(m+1)^{\prime}.
$$ 	
\end{proposition}
\begin{coro}\cite{JLM2}\label{c:18}
 The following relations hold in $U\left(\bar{R}^{[n]}\right)$
 $$
 \bar{R}_{12}^{[n-m]}(u / v) \mathcal{L}_1^{ \pm[n-m]}(u) \mathcal{L}_2^{ \pm[n-m]}(v)=\mathcal{L}_2^{ \pm[n-m]}(v) \mathcal{L}_1^{ \pm[n-m]}(u) \bar{R}_{12}^{[n-m]}(u / v),
 $$	
 $$
 \bar{R}_{12}^{[n-m]}\left(u_{+} / v_{-}\right) \mathcal{L}_1^{+[n-m]}(u) \mathcal{L}_2^{-[n-m]}(v)=\mathcal{L}_2^{-[n-m]}(v) \mathcal{L}_1^{+[n-m]}(u) \bar{R}_{12}^{[n-m]}\left(u_{-} / v_{+}\right).
 $$

 Now we can have more explicit description of the relations in $U\left(\bar{R}^{[n]}\right)$.
\end{coro}	
	\begin{proposition}	
		Suppose that $m+1 \leqslant j, k, l \leqslant(m+1)^{\prime}$ and $j \neq l^{\prime}$. Then, the following relations hold in $U\left(\bar{R}^{[n]}\right)$\\
		If j = l and $m+1 \leq l \leq n$, then
		$$
		\mathfrak{e}_{m j}^{ \pm}(u) \ell_{k l}^{\mp[n-m]}(v)=r^{-1} s^{-1} \frac{r u_{\mp}-s v_{ \pm}}{u_{\mp}-v_{ \pm}} \ell_{k j}^{\mp[n-m]}(v) \mathfrak{e}_{m l}^{ \pm}(u)-r^{-1} s^{-1} \frac{(r-s) u_{\mp}}{u_{\mp}-v_{ \pm}} \ell_{k j}^{\mp[n-m]}(v) \mathfrak{e}_{m j}^{\mp}(v),
		$$
		$$
		\mathfrak{e}_{m j}^{ \pm}(u) \ell_{k l}^{ \pm[n-m]}(v)=r^{-1} s^{-1} \frac{r u-s v}{u-v} \ell_{k j}^{ \pm[n-m]}(v) \mathfrak{e}_{m l}^{ \pm}(u)-r^{-1} s^{-1} \frac{(r-s) u}{u-v} \ell_{k j}^{ \pm[n-m]}(v) \mathfrak{e}_{m j}^{ \pm}(v);
		$$
		If j = l and $n+1 \leq l \leq(m+1)^{\prime}$, then
$$
\mathfrak{e}_{m j}^{ \pm}(u) \ell_{k l}^{\mp[n-m]}(v)=\frac{r u_{\mp}-s v_{ \pm}}{u_{\mp}-v_{ \pm}} \ell_{k j}^{\mp[n-m]}(v) \mathfrak{e}_{m l}^{ \pm}(u)-\frac{(r-s) u_{\mp}}{u_{\mp}-v_{ \pm}} \ell_{k j}^{\mp[n-m]}(v) \mathfrak{e}_{m j}^{\mp}(v),
$$
			$$
	\mathfrak{e}_{m j}^{ \pm}(u) \ell_{k l}^{ \pm[n-m]}(v)= \frac{r u-s v}{u-v} \ell_{k j}^{ \pm[n-m]}(v) \mathfrak{e}_{m l}^{ \pm}(u)- \frac{(r-s) u}{u-v} \ell_{k j}^{ \pm[n-m]}(v) \mathfrak{e}_{m j}^{ \pm}(v);
	$$
		If $j<l$ and $m+1 \leq l \leq n$, then
$$
\left[\mathfrak{e}_{m j}^{ \pm}(u), \ell_{k l}^{\mp[n-m]}(v)\right]=r^{-1} s^{-1} \frac{(r-s) v_{ \pm}}{u_{\mp}-v_{ \pm}} \ell_{k j}^{\mp[n-m]}(v) \mathfrak{e}_{m l}^{ \pm}(u)-r^{-1} s^{-1} \frac{(r-s) u_{\mp}}{u_{\mp}-v_{ \pm}} \ell_{k j}^{\mp[n-m]}(v) \mathfrak{e}_{m l}^{\mp}(v),
$$
	$$
\left[\mathfrak{e}_{m j}^{ \pm}(u), \ell_{k l}^{ \pm[n-m]}(v)\right]=r^{-1}s^{-1}\frac{\left(r-s\right) v}{u-v} \ell_{k j}^{ \pm[n-m]}(v) \mathfrak{e}_{m l}^{ \pm}(u)-r^{-1}s^{-1}\frac{\left(r-s\right) u}{u-v} \ell_{k j}^{ \pm[n-m]}(v) \mathfrak{e}_{m l}^{ \pm}(v);
$$	
	If $j<l$ and $n+1 \leq l \leq(m+1)^{\prime}$, then
	$$
	\left[\mathfrak{e}_{m j}^{ \pm}(u), \ell_{k l}^{\mp[n-m]}(v)\right]=\frac{(r-s) v_{ \pm}}{u_{\mp}-v_{ \pm}} \ell_{k j}^{\mp[n-m]}(v) \mathfrak{e}_{m l}^{ \pm}(u)-\frac{(r-s) u_{\mp}}{u_{\mp}-v_{ \pm}} \ell_{k j}^{\mp[n-m]}(v) \mathfrak{e}_{m l}^{\mp}(v),
	$$
	$$
	\left[\mathfrak{e}_{m j}^{ \pm}(u), \ell_{k l}^{ \pm[n-m]}(v)\right]=\frac{\left(r-s\right) v}{u-v} \ell_{k j}^{ \pm[n-m]}(v) \mathfrak{e}_{m l}^{ \pm}(u)-\frac{\left(r-s\right) u}{u-v} \ell_{k j}^{ \pm[n-m]}(v) \mathfrak{e}_{m l}^{ \pm}(v);
	$$
	If $j>l$ and $m+1 \leq l \leq n$, then
	$$
	\left[\mathfrak{e}_{m j}^{ \pm}(u), \ell_{k l}^{\mp[n-m]}(v)\right]=r^{-1} s^{-1} \frac{(r-s) u_{\mp}}{u_{\mp}-v_{ \pm}} \ell_{k j}^{\mp[n-m]}(v)\left(\mathfrak{e}_{m l}^{ \pm}(u)-\mathfrak{e}_{m l}^{\mp}(v)\right),
	$$	
	$$
	\left[\mathfrak{e}_{m j}^{ \pm}(u), \ell_{k l}^{ \pm[n-m]}(v)\right]=r^{-1}s^{-1}\frac{\left(r-s\right) u}{u-v} \ell_{k j}^{ \pm[n-m]}(v)\left(\mathfrak{e}_{m l}^{ \pm}(u)-\mathfrak{e}_{m l}^{ \pm}(v)\right);
	$$
	If $j>l$ and $n+1 \leq l \leq(m+1)^{\prime}$, then
	$$
	\left[\mathfrak{e}_{m j}^{ \pm}(u), \ell_{k l}^{\mp[n-m]}(v)\right]=\frac{(r-s) u_{\mp}}{u_{\mp}-v_{ \pm}} \ell_{k j}^{\mp[n-m]}(v)\left(\mathfrak{e}_{m l}^{ \pm}(u)-\mathfrak{e}_{m l}^{\mp}(v)\right),
	$$
	$$
	\left[\mathfrak{e}_{m j}^{ \pm}(u), \ell_{k l}^{ \pm[n-m]}(v)\right]=\frac{\left(r-s\right) u}{u-v} \ell_{k j}^{ \pm[n-m]}(v)\left(\mathfrak{e}_{m l}^{ \pm}(u)-\mathfrak{e}_{m l}^{ \pm}(v)\right).
	$$
	\begin{proof}
		It is sufficient to verify the relations for $m=1$,  the general case then follows by applying the homomorphism $\psi_m$. The
		calculations are similar for all the relations, and we show some of the essential computations. First, we verify that
	\begin{flalign*}
		\left[\mathfrak{e}_{m j}^{ \pm}(u), \ell_{k l}^{\mp[n-m]}(v)\right]=r^{-1} s^{-1} \frac{(r-s) v_{ \pm}}{u_{\mp}-v_{ \pm}} \ell_{k j}^{\mp[n-m]}(v) \mathfrak{e}_{m l}^{ \pm}(u)-r^{-1} s^{-1} \frac{(r-s) u_{\mp}}{u_{\mp}-v_{ \pm}} \ell_{k j}^{\mp[n-m]}(v) \mathfrak{e}_{m l}^{\mp}(v).
	\end{flalign*}	
		Denote
	\begin{flalign*}
		\bar{R}_{r s}(u)=r s \frac{u-1}{u r-s}\left(\sum_{\substack{1 \leq i \leq n-1 \\ i +1 \leq j \leq n}} E_{i i} \otimes E_{j j}+\sum_{\substack{1 \leq i \leq n-1 \\ i+1 \leq j \leq 2 n}} E_{i i} \otimes E_{j j}+\sum_{j=n+2}^{2 n} E_{n n} \otimes E_{j j}\right.
	\end{flalign*}
	\begin{flalign*}
\left.+\sum_{\substack{n+1 \leq j \leq 2 n-1 \\ j+1 \leq i \leq 2 n}} E_{i i} \otimes E_{j j}+\sum_{\substack{1 \leq i \leq n-1 \\ n+1 \leq j \leq 2 n-i}} E_{j j} \otimes E_{i i}\right),
	\end{flalign*}
We also need $\bar{R}_{11}(u)=(rs)^{-1}(\bar{R}_{r s}(u))^{op}$, where $op$ means flipping the tensor product: $a\otimes b\mapsto b\otimes a$.

		When $m+1 \leq k \leq n$ and $E_{jj}\otimes E_{ll}$ appears in $\bar{R}_{r s}(u)$,
		the defining relations give that
		\begin{equation}
	\begin{aligned}\label{pro19.1}
r s \frac{u_{ \pm}-v_{\mp}}{r u_{ \pm}-s^{-1} v_{\mp}} \ell_{1 j}^{ \pm}(u) \ell_{k l}^{\mp}(v)+\frac{(r-s) u_{ \pm}}{r u_{ \pm}-s v_{\mp}} \ell_{k j}^{ \pm}(u) \ell_{1 l}^{\mp}(v)\\
=r s \frac{u_{\mp}-v_{ \pm}}{r u_{\mp}-s v_{ \pm}} \ell_{k l}^{\mp}(v) \ell_{1 j}^{ \pm}(u)+\frac{(r-s) v_{ \pm}}{r u_{\mp}-s v_{ \pm}} \ell_{k j}^{\mp}(v) \ell_{1 l}^{ \pm}(u) .
\end{aligned}
\end{equation}
		Since $\ell_{k l}^{\mp}(v)=\ell_{k l}^{\mp[n-1]}(v)+\mathfrak{f}_{k 1}^{\mp}(v) \mathfrak{h}_1^{\mp}(v) \mathfrak{e}_{1 l}^{\mp}(v)$, we can write the left-hand  side of (\ref{pro19.1}) as
	\begin{flalign*}
		r s \frac{u_{ \pm}-v_{\mp}}{r u_{ \pm}-s v_{\mp}} \ell_{1 j}^{ \pm}(u) \ell_{k l}^{\mp[n-1]}(v)+r s \frac{u_{ \pm}-v_{\mp}}{r u_{ \pm}-s v_{\mp}} \ell_{1 j}^{ \pm}(u) \mathfrak{f}_{k 1}^{\mp}(v) \mathfrak{h}_1^{\mp}(v) \mathfrak{e}_{1 l}^{\mp}(v)+\frac{(r-s) u_{\mp}}{r u_{ \pm}-s v_{\mp}} \ell_{k j}^{ \pm}(u) \ell_{1 l}^{\mp}(v).
\end{flalign*}
By the defining relations, we have
	\begin{flalign*}
r s \frac{u_{ \pm}-v_{\mp}}{r u_{ \pm}-s v_{\mp}} \ell_{1 j}^{ \pm}(u) \ell_{k 1}^{\mp}(v)+\frac{(r-s) u_{ \pm}}{r u_{ \pm}-s v_{\mp}} \ell_{k j}^{ \pm}(u) \ell_{11}^{\mp}(v)\\
=r s \frac{u_{\mp}-v_{ \pm}}{r u_{\mp}-s v_{ \pm}} \ell_{k 1}^{\mp}(v) \ell_{1 j}^{ \pm}(u)+\frac{(r-s) u_{\mp}}{r u_{\mp}-s v_{ \pm}} \ell_{k j}^{\mp}(v) \ell_{11}^{ \pm}(u) .
\end{flalign*}
Hence, the left-hand side of (\ref{pro19.1}) equals to
	\begin{flalign*}
r s \frac{u_{ \pm}-v_{\mp}}{r u_{ \pm}-s v_{\mp}} \ell_{1 j}^{ \pm}(u) \ell_{k l}^{\mp[n-1]}(v)+r s \frac{u_{\mp}-v_{ \pm}}{r u_{\mp}-s v_{ \pm}} f_{k 1}^{\mp}(v) \ell_{11}^{\mp}(v) \ell_{1 j}^{ \pm}(u) \mathfrak{e}_{1 l}^{\mp}(v)
\end{flalign*}
	\begin{flalign*}
+\frac{\left(r-s\right) u_{\mp}}{r u_{\mp}-s v_{ \pm}} \ell_{k j}^{\mp}(v) \ell_{11}^{ \pm}(u) \mathfrak{e}_{1 l}^{\mp}(v) .
\end{flalign*}
Furthermore, using the relation
	\begin{flalign*}
\ell_{1 j}^{ \pm}(u) \ell_{11}^{\mp}(v)=rs\frac{u_{\mp}-v_{ \pm}}{r u_{\mp}-s v_{ \pm}} \ell_{11}^{\mp}(v) \ell_{1 j}^{ \pm}(u)+\frac{\left(r-s\right) u_{\mp}}{r u_{\mp}-s v_{ \pm}} \ell_{1 j}^{\mp}(v) \ell_{11}^{ \pm}(u),
\end{flalign*}
we turn the left-hand side of (\ref{pro19.1}) to the following form
	\begin{flalign*}
rs\frac{u_{ \pm}-v_{\mp}}{r u_{ \pm}-s^{-1} v_{\mp}} \ell_{1 j}^{ \pm}(u) \ell_{k l}^{\mp[n-1]}(v)+\mathfrak{f}_{k 1}^{\mp}(v) \ell_{1 j}^{ \pm}(u) \ell_{1 l}^{\mp}(v)+\frac{\left(r-s\right) u_{\mp}}{r u_{\mp}-s^{-1} v_{ \pm}} \ell_{k j}^{\mp[n-1]}(v) \ell_{11}^{ \pm}(u) \mathfrak{e}_{1 l}^{\mp}(v).
\end{flalign*}
For $j<l$, we have
	\begin{flalign*}
\ell_{1 j}^{ \pm}(u) \ell_{1 l}^{\mp}(v)=rs\frac{u_{\mp}-v_{ \pm}}{r u_{\mp}-s v_{ \pm}} \ell_{1 l}^{\mp}(v) \ell_{1 j}^{ \pm}(u)+\frac{\left(r-s\right) v_{ \pm}}{r u_{\mp}-s v_{ \pm}} \ell_{1 j}^{\mp}(v) \ell_{1 l}^{ \pm}(u),
\end{flalign*}
so that the left-hand side of (\ref{pro19.1})  becomes
	\begin{flalign*}
r s \frac{u_{ \pm}-v_{\mp}}{r u_{ \pm}-s v_{\mp}} \ell_{1 j}^{ \pm}(u) \ell_{k l}^{\mp[n-1]}(v)-\frac{u_{\mp}-v_{ \pm}}{r u_{\mp}-s v_{ \pm}} \ell_{k l}^{\mp[n-1]}(v) \ell_{1 j}^{ \pm}(u)
\end{flalign*}
	\begin{flalign*}
=r s \frac{(r-s) v_{ \pm}}{r u_{\mp}-s v_{ \pm}} \ell_{k j}^{\mp[n-1]}(v) \ell_{11}^{ \pm}(u) \mathfrak{e}_{1 l}^{ \pm}(u)-\frac{(r-s) u_{\mp}}{r u_{\mp}-s v_{ \pm}} \ell_{k j}^{\mp[n-1]}(v) \ell_{11}^{ \pm}(u) \mathfrak{e}_{1 l}^{\mp}(v).
\end{flalign*}
Finally, Corollary \ref{c:15} implies
	\begin{flalign*}
\begin{aligned}
	& \frac{u_{ \pm}-v_{\mp}}{r u_{ \pm}-sv_{\mp}} \ell_{11}^{ \pm}(u) \ell_{k l}^{\mp[n-1]}(v)=\frac{u_{\mp}-v_{ \pm}}{r u_{\mp}-s v_{ \pm}} \ell_{k l}^{\mp[n-1]}(v) \ell_{11}^{ \pm}(u), \\
	& \frac{u_{ \pm}-v_{\mp}}{r u_{ \pm}-s v_{\mp}} \ell_{11}^{ \pm}(u) \ell_{k j}^{\mp[n-1]}(v)=\frac{u_{\mp}-v_{ \pm}}{r u_{\mp}-s v_{ \pm}} \ell_{k j}^{\mp[n-1]}(v) \ell_{11}^{ \pm}(u).
\end{aligned}
	\end{flalign*}

Note that when $m+1 \leq k \leq n$ and $E_{jj}\otimes E_{ll}$ appears in $\bar{R}_{11}(u)$, $n+1 \leq k \leq(m+1)^{\prime}$ and $E_{jj}\otimes E_{ll}$ appears in $\bar{R}_{r s}(u)$, or $n+1 \leq k \leq(m+1)^{\prime}$ and $E_{jj}\otimes E_{ll}$ appears in $ \bar{R}_{11}(u)$, the computations are all
similar.

We now verify
	\begin{flalign*}
\mathfrak{e}_{m j}^{ \pm}(u) \ell_{k l}^{\mp[n-m]}(v)=r^{-1} s^{-1} \frac{r u_{\mp}-s v_{ \pm}}{u_{\mp}-v_{ \pm}} \ell_{k j}^{\mp[n-m]}(v) \mathfrak{e}_{m l}^{ \pm}(u)-r^{-1} s^{-1} \frac{(r-s) u_{\mp}}{u_{\mp}-v_{ \pm}} \ell_{k j}^{\mp[n-m]}(v) \mathfrak{e}_{m j}^{\mp}(v),
	\end{flalign*}
When $n+1 \leq k \leq(m+1)^{\prime}$ and $E_{jj}\otimes E_{ll}$ appears in $\bar{R}_{r s}(u),$ the defining relations imply that
\begin{equation}
	\begin{aligned}\label{pro19.2}
\frac{u_{ \pm}-v_{\mp}}{r u_{ \pm}-s^{-1} v_{\mp}} \ell_{1 l}^{ \pm}(u) \ell_{k l}^{\mp}(v)+\frac{(r-s) u_{ \pm}}{r u_{ \pm}-s v_{\mp}} \ell_{k l}^{ \pm}(u) \ell_{1 l}^{\mp}(v)=\ell_{k l}^{\mp}(v) \ell_{1 l}^{ \pm}(u) .
	\end{aligned}
\end{equation}
	Since $\ell_{k l}^{\mp}(v)=\ell_{k l}^{\mp[n-1]}(v)+\mathfrak{f}_{k 1}^{\mp}(v) \mathfrak{h}_1^{\mp}(v) \mathfrak{e}_{1 l}^{\mp}(v)$, we can write the left-hand side of (\ref{pro19.2}) as
	\begin{flalign*}
	\frac{u_{ \pm}-v_{\mp}}{r u_{ \pm}-s v_{\mp}} \ell_{1 l}^{ \pm}(u) \ell_{k l}^{\mp[n-1]}(v)+\frac{u_{ \pm}-v_{\mp}}{r u_{ \pm}-s v_{\mp}} \ell_{1 l}^{ \pm}(u) \mathfrak{f}_{k 1}^{\mp}(v) \mathfrak{h}_1^{\mp}(v) \mathfrak{e}_{1 l}^{\mp}(v)+\frac{(r-s) u_{\mp}}{r u_{ \pm}-s v_{\mp}} \ell_{k l}^{ \pm}(u) \ell_{1 l}^{\mp}(v).
	\end{flalign*}
By the defining relations, we have
	\begin{flalign*}
\frac{u_{ \pm}-v_{\mp}}{r u_{ \pm}-s v_{\mp}} \ell_{1 l}^{ \pm}(u) \ell_{k 1}^{\mp}(v)+\frac{(r-s) u_{ \pm}}{r u_{ \pm}-s v_{\mp}} \ell_{k l}^{ \pm}(u) \ell_{11}^{\mp}(v)=rs\frac{u_{\mp}-v_{ \pm}}{r u_{\mp}-s v_{ \pm}} \ell_{k 1}^{\mp}(v) \ell_{1 l}^{ \pm}(u)+\frac{(r-s) u_{\mp}}{r u_{\mp}-s v_{ \pm}} \ell_{k l}^{\mp}(v) \ell_{11}^{ \pm}(u),
	\end{flalign*}
Hence, the left-hand side of (\ref{pro19.2}) equals to
	\begin{flalign*}
\frac{u_{ \pm}-v_{\mp}}{r u_{ \pm}-s^{-1} v_{\mp}} \ell_{1 l}^{ \pm}(u) \ell_{k l}^{\mp[n-1]}(v)+r s \frac{u_{\mp}-v_{ \pm}}{r u_{\mp}-s^{-1} v_{ \pm}} f_{k 1}^{\mp}(v) \ell_{11}^{\mp}(v) \ell_{1 l}^{ \pm}(u) \mathfrak{e}_{1 l}^{\mp}(v)
	\end{flalign*}
	\begin{flalign*}
+\frac{(r-s) u_{\mp}}{r u_{\mp}-s^{-1} v_{ \pm}} \ell_{k l}^{\mp}(v) \ell_{11}^{ \pm}(u) \mathfrak{e}_{1 l}^{\mp}(v) .
	\end{flalign*}
Furthermore, using the relation
	\begin{flalign*}
\ell_{1 j}^{ \pm}(u) \ell_{11}^{\mp}(v)=r s \frac{u_{\mp}-v_{ \pm}}{r u_{\mp}-s v_{ \pm}} \ell_{11}^{\mp}(v) \ell_{1 l}^{ \pm}(u)+\frac{(r-s) u_{\mp}}{r u_{\mp}-s v_{ \pm}} \ell_{1 l}^{\mp}(v) \ell_{11}^{ \pm}(u),
	\end{flalign*}
we reduce the left-hand side of (\ref{pro19.2}) to the following
	\begin{flalign*}
\frac{u_{ \pm}-v_{\mp}}{r u_{ \pm}-s v_{\mp}} \ell_{1 l}^{ \pm}(u) \ell_{k l}^{\mp[n-1]}(v)+\mathfrak{f}_{k 1}^{\mp}(v) \ell_{1 l}^{ \pm}(u) \ell_{1 l}^{\mp}(v)+\frac{(r-s) u_{\mp}}{r u_{\mp}-s v_{ \pm}} \ell_{k l}^{\mp[n-1]}(v) \ell_{11}^{ \pm}(u) \mathfrak{e}_{1 l}^{\mp}(v).
	\end{flalign*}
For $j=l$, we have
	\begin{flalign*}
\ell_{1 l}^{ \pm}(u) \ell_{1 l}^{\mp}(v)=\ell_{1 l}^{\mp}(v) \ell_{1 l}^{ \pm}(u),
	\end{flalign*}
together with  $\ell_{k l}^{\mp}(v)=\ell_{k l}^{\mp[n-1]}(v)+\mathfrak{f}_{k 1}^{\mp}(v) \mathfrak{h}_1^{\mp}(v) \mathfrak{e}_{1 l}^{\mp}(v)$, the right-hand side of (\ref{pro19.2})  becomes
	\begin{flalign*}
\ell_{k l}^{\mp[n-1]}(v) \ell_{1 l}^{ \pm}(u)+\mathfrak{f}_{k 1}^{\mp}(v) \mathfrak{h}_1^{\mp}(v) \mathfrak{e}_{1 l}^{\mp}(v) \ell_{1 l}^{ \pm}(u)
	\end{flalign*}
	\begin{flalign*}
=\ell_{k l}^{\mp[n-1]}(v) \ell_{1 l}^{ \pm}(u)+\mathfrak{f}_{k 1}^{\mp}(v) \ell_{1 l}^{ \pm}(u) \mathfrak{h}_1^{\mp}(v) \mathfrak{e}_{1 l}^{\mp}(v).
	\end{flalign*}
Comparing the both side of (\ref{pro19.2}), we have
	\begin{flalign*}
 \frac{u_{ \pm}-v_{\mp}}{r u_{ \pm}-s v_{\mp}} \ell_{1 l}^{ \pm}(u) \ell_{k l}^{\mp[n-1]}(v)+\frac{(r-s) u_{\mp}}{r u_{\mp}-s v_{ \pm}} \ell_{k l}^{\mp[n-1]}(v) \ell_{11}^{ \pm}(u) \mathfrak{e}_{1 l}^{\mp}(v)=\ell_{k l}^{\mp[n-1]}(v) \ell_{1 l}^{ \pm}(u).
	\end{flalign*}
Finally, the Corollary \ref{c:15} implies
	\begin{flalign*}
\frac{u_{ \pm}-v_{\mp}}{r u_{ \pm}-s v_{\mp}} \ell_{11}^{ \pm}(u) \ell_{k l}^{\mp[n-1]}(v)=r s \frac{u_{\mp}-v_{ \pm}}{r u_{\mp}-s v_{ \pm}} \ell_{k l}^{\mp[n-1]}(v) \ell_{11}^{ \pm}(u).
	\end{flalign*}
So we can get this relation. When $m+1 \leq k \leq n$ and $E_{jj}\otimes E_{ll}$ appears in $\bar{R}_{rs}(u)$, $m+1 \leq k \leq n$ and $E_{jj}\otimes E_{ll}$ appears in $\bar{R}_{11}(u)$, or $n+1 \leq k \leq(m+1)^{\prime}$ and $E_{jj}\otimes E_{ll}$ appears in $\bar{R}_{11}(u)$, the computations are
similar and are omitted. Therefore the proposition is proved.
		\end{proof}	
	
\end{proposition}
By similar arguments we have derived the relations involving the generator series $f_{j i}^{ \pm}(u)$.

	\begin{proposition}	\label{p:20}
		Suppose that $m+1 \leqslant j, k, l \leqslant(m+1)^{\prime}$ and $j \neq k^{\prime}$.  Then, the following relations hold in $U\left(\bar{R}^{[n]}\right)$.
		
		If $j=k$ and $m+1 \leq k \leq n$, then
	\begin{flalign*}
		\mathfrak{f}_{j m}^{ \pm}(u) \ell_{j l}^{\mp[n-m]}(v)=r^{-1}s^{-1}\frac{u_{ \pm}-v_{\mp}}{r u_{ \pm}-s v_{\mp}} \ell_{j l}^{\mp[n-m]}(v) \mathfrak{f}_{j m}^{ \pm}(u)+r^{-1}s^{-1}\frac{\left(r-s\right) v_{\mp}}{r u_{ \pm}-s v_{\mp}} \mathfrak{f}_{j m}^{\mp}(v) \ell_{j l}^{\mp[n-m]}(v),
	\end{flalign*}
	\begin{flalign*}
\mathfrak{f}_{j m}^{ \pm}(u) \ell_{j l}^{ \pm[n-m]}(v)=r^{-1}s^{-1}\frac{u-v}{r u-s v} \ell_{j l}^{ \pm[n-m]}(u v) \mathfrak{f}_{j m}^{ \pm}(u)+r^{-1}s^{-1}\frac{\left(r-s\right) v}{r u-s v} \mathfrak{f}_{j m}^{ \pm}(v) \ell_{j l}^{ \pm[n-m]}(v) ;
	\end{flalign*}
If j = l and $n+1 \leq k \leq(m+1)^{\prime}$, then
	\begin{flalign*}
	\mathfrak{f}_{j m}^{ \pm}(u) \ell_{j l}^{\mp[n-m]}(v)=\frac{u_{ \pm}-v_{\mp}}{r u_{ \pm}-s v_{\mp}} \ell_{j l}^{\mp[n-m]}(v) \mathfrak{f}_{j m}^{ \pm}(u)+\frac{\left(r-s\right) v_{\mp}}{r u_{ \pm}-s v_{\mp}} \mathfrak{f}_{j m}^{\mp}(v) \ell_{j l}^{\mp[n-m]}(v),
\end{flalign*}
\begin{flalign*}
	\mathfrak{f}_{j m}^{ \pm}(u) \ell_{j l}^{ \pm[n-m]}(v)=\frac{u-v}{r u-s v} \ell_{j l}^{ \pm[n-m]}(u v) \mathfrak{f}_{j m}^{ \pm}(u)+\frac{\left(r-s\right) v}{r u-s v} \mathfrak{f}_{j m}^{ \pm}(v) \ell_{j l}^{ \pm[n-m]}(v) ;
\end{flalign*}
		If $j<l$ and $m+1 \leq k \leq n$, then
\begin{flalign*}
		\left[\mathfrak{f}_{j m}^{ \pm}(u), \ell_{k l}^{\mp[n-m]}(v)\right]=r^{-1}s^{-1}\frac{\left(r-s\right) v_{\mp}}{u_{ \pm}-v_{\mp}} \mathfrak{f}_{k m}^{\mp}(v) \ell_{j l}^{\mp[n-m]}(v)-r^{-1}s^{-1}\frac{\left(r-s\right) u_{ \pm}}{u_{ \pm}-v_{\mp}} \mathfrak{f}_{k m}^{ \pm}(u) \ell_{j l}^{\mp[n-m]}(v),
\end{flalign*}
\begin{flalign*}
\left[\mathfrak{f}_{j m}^{ \pm}(u), \ell_{k l}^{ \pm[n-m]}(v)\right]=r^{-1}s^{-1}\frac{\left(r-s\right) v}{u-v} \mathfrak{f}_{k m}^{ \pm}(v) \ell_{j l}^{ \pm[n-m]}(v)-r^{-1}s^{-1}\frac{\left(r-s\right) u}{u-v} \mathfrak{f}_{k m}^{ \pm}(u) \ell_{j l}^{ \pm[n-m]}(v) ;
\end{flalign*}
If j < l and $n+1 \leq k \leq(m+1)^{\prime}$, then
\begin{flalign*}
	\left[\mathfrak{f}_{j m}^{ \pm}(u), \ell_{k l}^{\mp[n-m]}(v)\right]=\frac{\left(r-s\right) v_{\mp}}{u_{ \pm}-v_{\mp}} \mathfrak{f}_{k m}^{\mp}(v) \ell_{j l}^{\mp[n-m]}(v)-\frac{\left(r-s\right) u_{ \pm}}{u_{ \pm}-v_{\mp}} \mathfrak{f}_{k m}^{ \pm}(u) \ell_{j l}^{\mp[n-m]}(v),
\end{flalign*}
\begin{flalign*}
	\left[\mathfrak{f}_{j m}^{ \pm}(u), \ell_{k l}^{ \pm[n-m]}(v)\right]=\frac{\left(r-s\right) v}{u-v} \mathfrak{f}_{k m}^{ \pm}(v) \ell_{j l}^{ \pm[n-m]}(v)-\frac{\left(r-s\right) u}{u-v} \mathfrak{f}_{k m}^{ \pm}(u) \ell_{j l}^{ \pm[n-m]}(v) ;
\end{flalign*}
		If $j>l$ and $m+1 \leq k \leq n$, then
		$$
		\left[\mathfrak{f}_{j m}^{ \pm}(u), \ell_{k l}^{\mp[n-m]}(v)\right]=r^{-1}s^{-1}\frac{\left(r-s\right) v_{\mp}}{u_{ \pm}-v_{\mp}}\left(\mathfrak{f}_{k m}^{\mp}(v)-\mathfrak{f}_{k m}^{ \pm}(u)\right) \ell_{j l}^{\mp[n-m]}(v),
		$$
		$$
		\left[\mathfrak{f}_{j m}^{ \pm}(u), \ell_{k l}^{ \pm[n-m]}(v)\right]=r^{-1}s^{-1}\frac{\left(r-s\right) v}{u-v}\left(\mathfrak{f}_{k m}^{ \pm}(v)-\mathfrak{f}_{k m}^{ \pm}(u)\right) \ell_{j l}^{ \pm[n-m]}(v) ;
		$$
		If j < l and $n+1 \leq k \leq(m+1)^{\prime}$, then
		$$
		\left[\mathfrak{f}_{j m}^{ \pm}(u), \ell_{k l}^{\mp[n-m]}(v)\right]=\frac{\left(r-s\right) v_{\mp}}{u_{ \pm}-v_{\mp}}\left(\mathfrak{f}_{k m}^{\mp}(v)-\mathfrak{f}_{k m}^{ \pm}(u)\right) \ell_{j l}^{\mp[n-m]}(v),
		$$
		$$
		\left[\mathfrak{f}_{j m}^{ \pm}(u), \ell_{k l}^{ \pm[n-m]}(v)\right]=\frac{\left(r-s\right) v}{u-v}\left(\mathfrak{f}_{k m}^{ \pm}(v)-\mathfrak{f}_{k m}^{ \pm}(u)\right) \ell_{j l}^{ \pm[n-m]}(v) .
		$$
	\end{proposition}
\subsection{Relations coming from Type $A$ quantum affine $R$ matrix}	
Set
$$
\mathcal{L}^{A \pm}(u)=\sum_{i, j=1}^n e_{i j} \otimes \ell_{i j}^{ \pm}(u).
$$
 Consider the $R$-matrix of two parameter quantum affine algebra $U_{r, s}(\widehat{\mathfrak{g l}}_n)$ \cite{JL2}

 $$
 \begin{aligned}
 	R(u) & =\sum_{i=1}^n E_{i i} \otimes E_{i i}+\frac{(1-u) r}{1-u r s^{-1}} \sum_{i>j} E_{i i} \otimes E_{j j}+\frac{(1-u) s^{-1}}{1-u r s^{-1}} \sum_{i<j} E_{i i} \otimes E_{j j} \\
 	& +\frac{1-r s^{-1}}{1-u r s^{-1}} \sum_{i>j} E_{i j} \otimes E_{j i}+\frac{\left(1-r s^{-1}\right) u}{1-r s^{-1} u} \sum_{i<j} E_{i j} \otimes E_{j i} ,
 \end{aligned}
 $$
 By comparing it with the $R$-matrix of $U_{r, s}(\mathrm{C}_n^{(1)})$, the summand containing indices $1, \ldots, n$ is
 $$
 \begin{aligned}
 	R'(u) & =\sum_{i=1}^n E_{i i} \otimes E_{i i}+\frac{(1-u) s^{-1}}{1-u rs^{-1}} \sum_{i>j} E_{i i} \otimes E_{j j}+\frac{(1-u) r}{1-u r s^{-1}} \sum_{i<j} E_{i i} \otimes E_{j j} \\
 	& +\frac{1-r s^{-1}}{1-u r s^{-1}} \sum_{i>j} E_{i j} \otimes E_{j i}+\frac{\left(1-r s^{-1}\right) u}{1-r s^{-1} u} \sum_{i<j} E_{i j} \otimes E_{j i} .
 \end{aligned}
 $$

 Note that under $r\mapsto s^{-1}$ and $s^{-1}\mapsto r$, the two $R$-matrices coincide. Therefore
there are type $A$ relations in the algebra $U\left(\bar{R}^{[n]}\right)$
 $$
 \begin{aligned}
 	R_A(u / v) \mathcal{L}_1^{A \pm}(u) \mathcal{L}_2^{A \pm}(v) & =\mathcal{L}_2^{A \pm}(v) \mathcal{L}_1^{A \pm}(u) R_A(u / v), \\
 	R_A\left(u_+/ v_-\right) \mathcal{L}_1^{A+}(u) \mathcal{L}_2^{A-}(v) & =\mathcal{L}_2^{A-}(v) \mathcal{L}_1^{A+}(u) R_A\left(u_-/ v_+\right) .
 \end{aligned}
 $$
 Hence, we get the following relations for the Gaussian generators $\mathfrak{h}_i^{ \pm}(u)$ with $i=1, \ldots, n$ and $\mathcal{X}_i^{ \pm}(u)$ with $i=1, \ldots, n-1$ by using \cite{JL2}.
 	\begin{proposition}	
 In the algebra $U\left(\bar{R}^{[n]}\right)$, we have
 \begin{align*}
& \mathfrak{h}_i^{ \pm}(u) \mathfrak{h}_j^{ \pm}(v)=\mathfrak{h}_j^{ \pm}(v) \mathfrak{h}_i^{ \pm}(u), \quad \mathfrak{h}_i^{ \pm}(u) \mathfrak{h}_i^{\mp}(v)=\mathfrak{h}_i^{\mp}(v) \mathfrak{h}_i^{ \pm}(u), \\
& \frac{u_{ \pm}-v_{\mp}}{s^{-1} u_{ \pm}-r^{-1} v_{\mp}} \mathfrak{h}_i^{ \pm}(u) \mathfrak{h}_j^{\mp}(v)=\frac{u_{\mp}-v_{ \pm}}{s^{-1}u_{\mp}-r^{-1} v_{ \pm}} \mathfrak{h}_j^{\mp}(v) \mathfrak{h}_i^{ \pm}(u) \quad \text { for } \quad i<j\\
& \mathfrak{h}_i^{ \pm}(u) \mathcal{X}_j^{+}(v)=\frac{u_{\mp}-v}{s^{-\left(\epsilon_i, \alpha_j\right)} u_{\mp}-r^{-\left(\epsilon_i, \alpha_j\right)} v} \mathcal{X}_j^{+}(v) \mathfrak{h}_i^{ \pm}(u),\\
& \mathfrak{h}_i^{ \pm}(u) \mathcal{X}_j^{-}(v)=\frac{s^{-\left(\epsilon_i, \alpha_j\right)} u_{ \pm}-r^{-\left(\epsilon_i, \alpha_j\right)} v}{u_{ \pm}-v} \mathcal{X}_j^{-}(v) \mathfrak{h}_i^{ \pm}(u), \\
& (u-s^{\mp(\alpha_i, \alpha_j)} v) \mathcal{X}_i^{ \pm}(u(r s^{-1})^{\frac{i}{2}}) \mathcal{X}_j^{ \pm}(v(r s^{-1})^{\frac{j}{2}})
 =(s^{\mp(\alpha_i, \alpha_j)} u-v) \mathcal{X}_j^{ \pm}(v(r s^{-1})^{\frac{j}{2}}) \mathcal{X}_i^{ \pm}(u(r s^{-1})^{\frac{i}{2}}), \\
& \left[\mathcal{X}_i^{+}(u), \mathcal{X}_j^{-}(v)\right]\\
&=\delta_{i j}\left(s^{-1}-r^{-1}\right)\left(\delta\left(u_{-} / v_{+}\right) \mathfrak{h}_i^{-}\left(v_{+}\right)^{-1} \mathfrak{h}_{i+1}^{-}\left(v_{+}\right)-\delta\left(u_{+} / v_{-}\right) \mathfrak{h}_i^{+}\left(u_{+}\right)^{-1} \mathfrak{h}_{i+1}^{+}\left(u_{+}\right)\right).
 \end{align*}
 together with the Serre relations for the series $\mathcal{X}_i^{ \pm}(u)$.
  	\end{proposition}
  		\begin{remark}\label{r:22}	
 Consider the inverse matrices $\mathcal{L}^{ \pm}(u)^{-1}=\left[\ell_{i j}^{ \pm}(u)^{\prime}\right]_{i, j=1, \ldots, 2 n}$, we have
 $$
 \mathcal{L}_1^{ \pm}(u)^{-1} \mathcal{L}_2^{ \pm}(v)^{-1} \bar{R}^{[n]}(u / v)=\bar{R}^{[n]}(u / v) \mathcal{L}_2^{ \pm}(v)^{-1} \mathcal{L}_1^{ \pm}(u)^{-1},
 $$
 $$
 \mathcal{L}_2^{-}(v)^{-1} \mathcal{L}_1^{+}(u)^{-1} \bar{R}^{[n]}\left(u_{+} / v_{-}\right)=\bar{R}^{[n]}\left(u_{-} / v_{+}\right) \mathcal{L}_2^{-}(v)^{-1} \mathcal{L}_1^{+}(u)^{-1} .
 $$
 Hence, the coefficients of the series $\ell^{\pm}_{ij}(u)'$ belong to the algebra $U\left(\bar{R}^{[n]}\right)$
 and obey the defining relations of $U_{r, s}(\widehat{\mathfrak{g l}}_n)$, where $i, j=n^{\prime}, \ldots, 1^{\prime}$.  In particular, the Gauss
 decomposition for the matrix $\left[\ell_{i j}^{ \pm}(u)^{\prime}\right]_{i, j=n^{\prime}, \ldots, 1^{\prime}}$ comes from that of
 $\mathcal{L}^{ \pm}(u)$ by taking inverse.
		  	\end{remark}
When $n=1$, we have
$$
\bar{R}^{[1]}(u)=\sum_{i=n}^{n+1} e_{i i} \otimes e_{i i}+\frac{u-1}{r s^{-1} u-r^{-1} s} \sum_{i \neq j} e_{i i} \otimes e_{j j}
$$
$$
	 +\frac{\left(r s^{-1}-r^{-1} s\right) u}{r s^{-1} u-r^{-1} s} e_{n, n+1} \otimes e_{n+1, n}+\frac{r s^{-1}-r^{-1} s}{r s^{-1} u-r^{-1} s} e_{n+1, n} \otimes e_{n, n+1}.
$$
and so it coincides with the $R$-matrix associated with $U_{r s^{-1}}(\widehat{\mathfrak{g l}}_n)$.  Therefore, we have following proposition.
 	\begin{proposition}\label{p:23}	
 		 The following relations hold in the algebra $U\left(\bar{R}^{[n]}\right)$:
 		 $$
 		 \begin{aligned}
 		 	\mathfrak{h}_i^{ \pm}(u) \mathfrak{h}_j^{ \pm}(v) & =\mathfrak{h}_j^{ \pm}(v) \mathfrak{h}_i^{ \pm}(u), & & i, j=n, n+1, \\
 		 	\mathfrak{h}_i^{ \pm}(u) \mathfrak{h}_i^{\mp}(v) & =\mathfrak{h}_i^{\mp}(v) \mathfrak{h}_i^{ \pm}(u), & & i=n, n+1,
 		 \end{aligned}
 		 $$
 		 $$
 		 \frac{u_{ \pm}-v_{\mp}}{r s^{-1} u_{ \pm}-r^{-1} s v_{\mp}} \mathfrak{h}_n^{ \pm}(u) \mathfrak{h}_{n+1}^{\mp}(v)=\frac{u_{\mp}-v_{ \pm}}{r s^{-1} u_{\mp}-r^{-1} s v_{ \pm}} \mathfrak{h}_{n+1}^{\mp}(v) \mathfrak{h}_n^{ \pm}(u),
 		 $$
 		 $$
 		 \mathfrak{h}_n^{ \pm}(u) \mathcal{X}_n^{+}(v)=\frac{u_{\mp}-v}{r s^{-1} u_{\mp}-r^{-1} s v} \mathcal{X}_n^{+}(v) \mathfrak{h}_n^{ \pm}(u),
 		 $$
 		 $$
 		 \mathfrak{h}_{n+1}^{ \pm}(u) \mathcal{X}_n^{+}(v)=\frac{u_{\mp}-v}{r^{-1} s u_{\mp}-r s^{-1} v} \mathcal{X}_n^{+}(v) \mathfrak{h}_{n+1}^{ \pm}(u),
 		 $$
 		 $$
 		 \mathfrak{h}_n^{ \pm}(u) \mathcal{X}_n^{-}(v)=\frac{r s^{-1} u_{ \pm}-r^{-1} s v}{u_{ \pm}-v} \mathcal{X}_n^{-}(v) \mathfrak{h}_n^{ \pm}(u),
 		 $$
 		 $$
 		 \mathfrak{h}_{n+1}^{ \pm}(u) \mathcal{X}_n^{-}(v)=\frac{r^{-1} s u_{ \pm}-r s^{-1} v}{u_{ \pm}-v} \mathcal{X}_n^{-}(v) \mathfrak{h}_{n+1}^{ \pm}(u),
 		 $$
 		 $$
 		 (u-(r s^{-1})^{\frac{ \pm(\alpha_n, \alpha_n)}{2}} v) \mathcal{X}_n^{ \pm}(u) \mathcal{X}_n^{ \pm}(v)=((r s^{-1})^{\frac{ \pm(\alpha_n, \alpha_n)}{2}} u-v) \mathcal{X}_n^{ \pm}(v) \mathcal{X}_n^{ \pm}(u),
 		 $$
 		 $$
 		 \left[\mathcal{X}_n^{+}(u), \mathcal{X}_n^{-}(v)\right]=\left(r s^{-1}-r^{-1} s\right)\left(\delta\left(\frac{u_{-}}{v_{+}}\right) \mathfrak{h}_n^{-}\left(v_{+}\right)^{-1} \mathfrak{h}_{n+1}^{-}\left(v_{+}\right)-\delta\left(\frac{u_{+}}{v_{-}}\right) \mathfrak{h}_n^{+}\left(u_{+}\right)^{-1} \mathfrak{h}_{n+1}^{+}\left(u_{+}\right)\right).
 		 $$
 			\end{proposition}	
		\section{ Drinfeld-type relations in the algebras $\boldsymbol{U}\left(\bar{R}^{[n]}\right)$}
			\begin{theorem}
				The following relations between the Gaussian generators hold in the algebra $\boldsymbol{U}\left(\bar{R}^{[n]}\right)$.
				
				For the relations involving $\mathfrak{h}_i^{ \pm}(u),$  we	have $$
				\mathfrak{h}_{i, 0}^{+} \mathfrak{h}_{i, 0}^{-}=\mathfrak{h}_{i, 0}^{-} \mathfrak{h}_{i, 0}^{+}=1 \text {, and } \mathfrak{h}_{n, 0}^{+} \mathfrak{h}_{n+1,0}^{+}=1 \text {, }
				$$
$$\mathfrak{h}_i^{ \pm}(u) \mathfrak{h}_j^{ \pm}(v)=\mathfrak{h}_j^{ \pm}(v) \mathfrak{h}_i^{ \pm}(u),$$
$$
\mathfrak{h}_i^{ \pm}(u) \mathfrak{h}_i^{\mp}(v)=\mathfrak{h}_i^{\mp}(v) \mathfrak{h}_i^{ \pm}(u),
$$
		$$
		\frac{u_{ \pm}-v_{\mp}}{r u_{ \pm}-s v_{\mp}} \mathfrak{h}_i^{ \pm}(u) \mathfrak{h}_j^{\mp}(v)=\frac{u_{\mp}-v_{ \pm}}{r u_{\mp}-s v_{ \pm}} \mathfrak{h}_j^{\mp}(v) \mathfrak{h}_i^{ \pm}(u),
		$$
		for $i<j$ and $i \neq n$, and 	
$$
\frac{u_{ \pm}-v_{\mp}}{r s^{-1} u_{ \pm}-r^{-1} s v_{\mp}} \mathfrak{h}_n^{ \pm}(u) \mathfrak{h}_{n+1}^{\mp}(v)=\frac{u_{\mp}-v_{ \pm}}{r s^{-1} u_{\mp}-r^{-1} s v_{ \pm}} \mathfrak{h}_{n+1}^{\mp}(v) \mathfrak{h}_n^{ \pm}(u) .
$$
The relations involving $\mathfrak{h}_i^{ \pm}(u)$ and $\mathcal{X}_j^{ \pm}(v)$ are
 $$
\mathfrak{h}_i^{ \pm}(u) \mathcal{X}_j^{+}(v)=\frac{u_{\mp}-v}{s^{-\left(\epsilon_i, \alpha_j\right)} u_{\mp}-r^{-\left(\epsilon_i, \alpha_j\right)} v} \mathcal{X}_j^{+}(v) \mathfrak{h}_i^{ \pm}(u),
$$
$$
\mathfrak{h}_i^{ \pm}(u) \mathcal{X}_j^{-}(v)=\frac{s^{-\left(\epsilon_i, \alpha_j\right)} u_{ \pm}-r^{-\left(\epsilon_i, \alpha_j\right)} v}{u_{ \pm}-v} \mathcal{X}_j^{-}(v) \mathfrak{h}_i^{ \pm}(u),
$$
for $i=1, \ldots, n$, and $j=1, \ldots, n-1$, and
$$
\mathfrak{h}_i^{ \pm}(u) \mathcal{X}_n^{\epsilon}(v)=(rs)^{-\epsilon}\mathcal{X}_n^{\epsilon}(v) \mathfrak{h}_i^{ \pm}(u),
$$
for $i=1,2, \cdots ,n-1$ together with
$$
\mathfrak{h}_{n+1}^{ \pm}(u) \mathcal{X}_n^{+}(v)=\frac{u_{\mp}-v}{r^{-1} s u_{\mp}-r s^{-1} v} \mathcal{X}_n^{+}(v) \mathfrak{h}_{n+1}^{ \pm}(u),
$$
$$
\mathfrak{h}_{n+1}^{ \pm}(u) \mathcal{X}_n^{-}(v)=\frac{r^{-1} s u_{ \pm}-r s^{-1} v}{u_{ \pm}-v} \mathcal{X}_n^{-}(v) \mathfrak{h}_{n+1}^{ \pm}(u),
$$
$$
\begin{aligned}
	\mathfrak{h}_{n+1}^{ \pm}(u)^{-1} \mathcal{X}_{n-1}^{+}(u) \mathfrak{h}_{n+1}^{ \pm}(u) & =\frac{r^{-1} u_{\mp}-s^{-1} v}{r^{-1} s u_{\mp}-r s^{-1} v} \mathcal{X}_{n-1}^{ \pm}(u), \\
	\mathfrak{h}_{n+1}^{ \pm}(u)^{-1} \mathcal{X}_{n-1}^{-}(u) \mathfrak{h}_{n+1}^{ \pm}(u) & =\frac{r^{-1} u_{ \pm}-s^{-1} v}{r^{-1} s u_{ \pm}-r s^{-1} v} \mathcal{X}_{n-1}^{ \pm}(u),
\end{aligned}
$$
while $(\epsilon=\pm)$
$$
\begin{aligned}
	\mathfrak{h}_{n+1}^{ \pm}(u) \mathcal{X}_i^{\epsilon}(v)=\mathcal{X}_i^{\epsilon}(v) \mathfrak{h}_{n+1}^{ \pm}(u),
\end{aligned}
$$
for $1 \leqslant i \leqslant n-2$. For the relations involving $\mathcal{X}_i^{ \pm}(u)$,  we have
\begin{align*}
(u-s^{\mp(\alpha_i, \alpha_j)} v) \mathcal{X}_i^{ \pm}(u(r s^{-1})^{\frac{i}{2}}) \mathcal{X}_j^{ \pm}(v(r s^{-1})^{\frac{j}{2}})
=(s^{\mp(\alpha_i, \alpha_j)} u-v) \mathcal{X}_j^{ \pm}(v(r s^{-1})^{\frac{j}{2}}) \mathcal{X}_i^{ \pm}(u(r s^{-1})^{\frac{i}{2}}),
\end{align*}
for $i, j=1, \ldots, n-1$,
$$
\mathcal{X}_i^{ \pm}(u) \mathcal{X}_n^{ \pm}(u)=\mathcal{X}_n^{ \pm}(u) \mathcal{X}_i^{ \pm}(u),
$$
for $i=1,2, \cdots, n-2$, together with
\begin{flalign*}
	\left(u r s^{-1}-r^{-1} s v\right)^{ \pm 1} \mathcal{X}_{n-1}^{ \pm}(u) \mathcal{X}_n^{ \pm}(v)=(r s)^{ \pm 1}(u-v)^{ \pm 1} \mathcal{X}_n^{ \pm}(v) \mathcal{X}_{n-1}^{ \pm}(u),
\end{flalign*}
$$
(u-(r s^{-1})^{\frac{ \pm(\alpha_n, \alpha_n)}{2}} v) \mathcal{X}_n^{ \pm}(u) \mathcal{X}_n^{ \pm}(v)=((r s^{-1})^{\frac{ \pm(\alpha_n, \alpha_n)}{2}} u-v) \mathcal{X}_n^{ \pm}(v) \mathcal{X}_n^{ \pm}(u),
$$
and
 $$
\left[\mathcal{X}_i^{+}(u), \mathcal{X}_j^{-}(v)\right]=\delta_{i j}\left(s^{-1}-r^{-1}\right)\left(\delta\left(u_{-} / v_{+}\right) \mathfrak{h}_i^{-}\left(v_{+}\right)^{-1} \mathfrak{h}_{i+1}^{-}\left(v_{+}\right)-\delta\left(u_{+} / v_{-}\right) \mathfrak{h}_i^{+}\left(u_{+}\right)^{-1} \mathfrak{h}_{i+1}^{+}\left(u_{+}\right)\right).
$$
together with the Serre relations.

	\end{theorem}
			\begin{proof}
			We only need to verify the relations complementary to those produced in 4.4 and 4.5.   We first prove $$
			\frac{u_{ \pm}-v_{\mp}}{r u_{ \pm}-s v_{\mp}} \mathfrak{h}_i^{ \pm}(u) \mathfrak{h}_j^{\mp}(v)=\frac{u_{\mp}-v_{ \pm}}{r u_{\mp}-s v_{ \pm}} \mathfrak{h}_j^{\mp}(v) \mathfrak{h}_i^{ \pm}(u).
			$$
			for $j=n+1.$
			
			 By using Corollary \ref{c:15} for $i<n$, we deduce that
\begin{flalign}
\begin{aligned}\label{thm24.1}
				&\frac{u_{ \pm}-v_{\mp}}{r u_{ \pm}-s v_{\mp}} \mathfrak{h}_i^{ \pm}(u)\left(\mathfrak{h}_{n+1}^{\mp}(v)+\right. \left.\mathfrak{f}_{n+1, n}^{\mp}(v) \mathfrak{h}_n^{\mp}(v) \mathfrak{e}_{n, n+1}^{\mp}(v)\right) \\
				& =\frac{u_{\mp}-v_{ \pm}}{r u_{\mp}-s v_{ \pm}}\left(\mathfrak{h}_{n+1}^{ \pm}(v)+\mathfrak{f}_{n+1, n}^{\mp}(v) \mathfrak{h}_n^{\mp}(v) \mathfrak{e}_{n, n+1}^{\mp}(v)\right) \mathfrak{h}_i^{ \pm}(u),
\end{aligned}
\end{flalign}
and
\begin{flalign*}
\frac{u_{ \pm}-v_{\mp}}{r u_{ \pm}-s v_{\mp}} \mathfrak{h}_i^{ \pm}(u) \mathfrak{f}_{n+1, n}^{\mp}(v) \mathfrak{h}_n^{\mp}(v)=rs\frac{u_{\mp}-v_{ \pm}}{r u_{\mp}-s v_{ \pm}} \mathfrak{f}_{n+1, n}^{\mp}(v) \mathfrak{h}_n^{\mp}(v) \mathfrak{h}_i^{ \pm}(u) ,
\end{flalign*}
Hence, the left hand side of (\ref{thm24.1}) equals to
\begin{flalign*}
\frac{u_{ \pm}-v_{\mp}}{r u_{ \pm}-s v_{\mp}} \mathfrak{h}_i^{ \pm}(u) \mathfrak{h}_{n+1}^{\mp}(v)+rs\frac{u_{\mp}-v_{ \pm}}{r u_{\mp}-s v_{ \pm}} \mathfrak{f}_{n+1, n}^{\mp}(v) \mathfrak{h}_n^{\mp}(v) \mathfrak{h}_i^{ \pm}(u) \mathfrak{e}_{n, n+1}^{\mp}(v).
\end{flalign*}
Using Corollary \ref{c:15}
\begin{flalign*}
\frac{u_{ \pm}-v_{\mp}}{r u_{ \pm}-s v_{\mp}} \mathfrak{h}_i^{ \pm}(u) \mathfrak{h}_n^{\mp}(v)=\frac{u_{\mp}-v_{ \pm}}{r u_{\mp}-s v_{ \pm}} \mathfrak{h}_n^{\mp}(v) \mathfrak{h}_i^{ \pm}(u),
\end{flalign*}
so that the left hand side of (\ref{thm24.1}) can be written as
\begin{flalign*}
\frac{u_{ \pm}-v_{\mp}}{r u_{ \pm}-s v_{\mp}} \mathfrak{h}_i^{ \pm}(u) \mathfrak{h}_{n+1}^{\mp}(v)+rs\frac{u_{ \pm}-v_{\mp}}{r u_{ \pm}-s v_{\mp}} \mathfrak{f}_{n+1, n}^{\mp}(v) \mathfrak{h}_i^{ \pm}(u) \mathfrak{h}_n^{\mp}(v) \mathfrak{e}_{n, n+1}^{\mp}(v).
\end{flalign*}
Using Corollary \ref{c:15} once again, we find that
\begin{flalign*}
rs\frac{u_{ \pm}-v_{\mp}}{r u_{ \pm}-s v_{\mp}} \mathfrak{h}_i^{ \pm}(u) \mathfrak{h}_n^{\mp}(v) \mathfrak{e}_{n, n+1}^{\mp}(v)=\frac{u_{\mp}-v_{ \pm}}{r u_{\mp}-s v_{ \pm}} \mathfrak{h}_n^{\mp}(v) \mathfrak{e}_{n, n+1}^{\mp}(v) \mathfrak{h}_i^{ \pm}(u),
\end{flalign*}
for $i=1,2, \ldots, n-1$, so the left-hand side of (\ref{thm24.1}) takes the form
\begin{flalign*}
\frac{u_{ \pm}-v_{\mp}}{r u_{ \pm}-s v_{\mp}} \mathfrak{h}_i^{ \pm}(u) \mathfrak{h}_{n+1}^{\mp}(v)+\frac{u_{\mp}-v_{ \pm}}{r u_{\mp}-s v_{ \pm}} \mathfrak{f}_{n+1, n}^{\mp}(v) \mathfrak{h}_n^{\mp}(v) \mathfrak{e}_{n, n+1}^{\mp}(v) \mathfrak{h}_i^{ \pm}(u).
\end{flalign*}
which proved the relation with $j=n+1$.
The relations involving $\mathfrak{h}_{n+1}^{ \pm}(v)$ and $\mathcal{X}_i^{ \pm}(u)$, $i=1,2, \ldots, n-2$  are consequences of
the following relations:
\begin{flalign*}
\mathfrak{e}_{i, i+1}^{ \pm}(u) \mathfrak{h}_{n+1}^{ \pm}(v)=\mathfrak{h}_{n+1}^{ \pm}(v) \mathfrak{e}_{i, i+1}^{ \pm}(u), \quad \mathfrak{e}_{i, i+1}^{ \pm}(u) \mathfrak{h}_{n+1}^{\mp}(v)=\mathfrak{h}_{n+1}^{\mp}(v) \mathfrak{e}_{i, i+1}^{ \pm}(u),
\end{flalign*}
\begin{flalign*}
\mathfrak{f}_{i+1, i}^{ \pm}(u) \mathfrak{h}_{n+1}^{ \pm}(v)=\mathfrak{h}_{n+1}^{ \pm}(v) \mathfrak{f}_{i+1, i}^{ \pm}(u), \quad \mathfrak{f}_{i+1, i}^{ \pm}(u) \mathfrak{h}_{n+1}^{\mp}(v)=\mathfrak{h}_{n+1}^{\mp}(v) \mathfrak{f}_{i+1, i}^{ \pm}(u) .
\end{flalign*}
We will verify the second relation. Corollary \ref{c:15} implies that
\begin{flalign}
\begin{aligned}\label{thm24.2}
	& \frac{u_{ \pm}-v_{\mp}}{r u_{ \pm}-s v_{\mp}} \mathfrak{h}_i^{ \pm}(u) \mathfrak{e}_{i, i+1}^{ \pm}(u)\left(\mathfrak{h}_{n+1}^{\mp}(v)+\mathfrak{f}_{n+1, n}^{\mp}(v) \mathfrak{h}_n^{\mp}(v) \mathfrak{e}_{n, n+1}^{\mp}(v)\right) \\
	&=\frac{u_{\mp}-v_{ \pm}}{r u_{\mp}-s v_{ \pm}}\left(\mathfrak{h}_{n+1}^{\mp}(v)+\mathfrak{f}_{n+1, n}^{\mp}(v) \mathfrak{h}_n^{\mp}(v) \mathfrak{e}_{n, n+1}^{\mp}(v)\right) \mathfrak{h}_i^{ \pm}(u) \mathfrak{e}_{i, i+1}^{ \pm}(u) .
\end{aligned}
\end{flalign}
and also the following relations
\begin{flalign*}
\frac{u_{ \pm}-v_{\mp}}{r u_{ \pm}-s v_{\mp}} \mathfrak{h}_i^{ \pm}(u) \mathfrak{e}_{i, i+1}^{ \pm}(u) \mathfrak{f}_{n+1, n}^{\mp}(v) \mathfrak{h}_n^{\mp}(v)=rs\frac{u_{\mp}-v_{ \pm}}{r u_{\mp}-s v_{ \pm}} \mathfrak{f}_{n+1, n}^{\mp}(v) \mathfrak{h}_n^{\mp}(v) \mathfrak{h}_i^{ \pm}(u) \mathfrak{e}_{i, i+1}^{ \pm}(u),
\end{flalign*}
\begin{flalign*}
\frac{u_{ \pm}-v_{\mp}}{r u_{ \pm}-s v_{\mp}} \mathfrak{h}_i^{ \pm}(u) \mathfrak{e}_{i, i+1}^{ \pm}(u) \mathfrak{h}_n^{\mp}(v)=\frac{u_{\mp}-v_{ \pm}}{r u_{\mp}-s v_{ \pm}} \mathfrak{h}_n^{\mp}(v) \mathfrak{h}_i^{ \pm}(u) \mathfrak{e}_{i, i+1}^{ \pm}(u),
\end{flalign*}
\begin{flalign*}
rs\frac{u_{ \pm}-v_{\mp}}{r u_{ \pm}-s v_{\mp}} \mathfrak{h}_i^{ \pm}(u) \mathfrak{e}_{i, i+1}^{ \pm}(u) \mathfrak{h}_n^{\mp}(v) \mathfrak{e}_{n, n+1}^{\mp}(v)=\frac{u_{\mp}-v_{ \pm}}{r u_{\mp}-s v_{ \pm}} \mathfrak{h}_n^{\mp}(v) \mathfrak{e}_{n, n+1}^{\mp}(v) \mathfrak{h}_i^{ \pm}(u) \mathfrak{e}_{i, i+1}^{ \pm}(u) ,
\end{flalign*}
Thus, the left-hand side of (\ref{thm24.2}) equals to
\begin{flalign*}
\frac{u_{ \pm}-v_{\mp}}{r u_{ \pm}-s v_{\mp}} \mathfrak{h}_i^{ \pm}(u) \mathfrak{e}_{i, i+1}^{ \pm}(u) \mathfrak{h}_{n+1}^{\mp}(v)+\frac{u_{\mp}-v_{ \pm}}{r u_{\mp}-s v_{ \pm}} \mathfrak{f}_{n+1, n}^{\mp}(v) \mathfrak{h}_n^{\mp}(v) \mathfrak{e}_{n, n+1}^{\mp}(v) \mathfrak{h}_i^{ \pm}(u) \mathfrak{e}_{i, i+1}^{ \pm}(u),
\end{flalign*}
So we get
\begin{flalign*}
\frac{u_{ \pm}-v_{\mp}}{r u_{ \pm}-s v_{\mp}} \mathfrak{h}_i^{ \pm}(u) \mathfrak{e}_{i, i+1}^{ \pm}(u) \mathfrak{h}_{n+1}^{\mp}(v)=\frac{u_{\mp}-v_{ \pm}}{r u_{\mp}-s v_{ \pm}} \mathfrak{h}_{n+1}^{\mp}(v) \mathfrak{h}_i^{ \pm}(u) \mathfrak{e}_{i, i+1}^{ \pm}(u),
\end{flalign*}
Using the relation $
\frac{u_{ \pm}-v_{\mp}}{r u_{ \pm}-s v_{\mp}} \mathfrak{h}_i^{ \pm}(u) \mathfrak{h}_{n+1}^{\mp}(v)=\frac{u_{\mp}-v_{ \pm}}{r u_{\mp}-s v_{ \pm}} \mathfrak{h}_{n+1}^{\mp}(v) \mathfrak{h}_i^{ \pm}(u),
$
we get the relation $$\mathfrak{e}_{i, i+1}^{ \pm}(u) \mathfrak{h}_{n+1}^{\mp}(v)=\mathfrak{h}_{n+1}^{\mp}(v) \mathfrak{e}_{i, i+1}^{ \pm}(u).$$
Next we check the relations involving $\mathfrak{h}_{i}^{ \pm}(v)$ and $\mathcal{X}_n^{ \pm}(u)$ $i=1,2, \ldots, n-1$. In fact, 
$$
r s \mathfrak{e}_{n, n+1}^{ \pm}(u) \mathfrak{h}_i^{ \pm}(v)=\mathfrak{h}_i^{ \pm}(v) \mathfrak{e}_{n, n+1}^{ \pm}(u), \quad r s \mathfrak{e}_{n, n+1}^{ \pm}(u) \mathfrak{h}_i^{\mp}(v)=\mathfrak{h}_i^{\mp}(v) \mathfrak{e}_{n, n+1}^{ \pm}(u),
$$
$$
r s \mathfrak{f}_{n+1, n}^{ \pm}(u) \mathfrak{h}_i^{ \pm}(v)=\mathfrak{h}_i^{ \pm}(v) \mathfrak{f}_{n+1, n}^{ \pm}(u), \quad r s \mathfrak{f}_{n+1, n}^{ \pm}(u) \mathfrak{h}_i^{\mp}(v)=\mathfrak{h}_i^{\mp}(v) \mathfrak{f}_{n+1, n}^{ \pm}(u).
$$
Finally we will verify the last relation, and the others can be proved similarly. Corollary \ref{c:15} implies the following relations
\begin{align*}
	&\frac{u_{ \mp}-v_{\pm}}{r u_{ \mp}-s v_{\pm}}\mathfrak{h}_i^{\mp}(u)\left(\mathfrak{f}_{n+1, n-1}^{ \pm}(v) \mathfrak{h}_{n-1}^{ \pm}(v) \mathfrak{e}_{n-1, n}^{ \pm}(v)+\mathfrak{f}_{n+1, n}^{ \pm}(v) \mathfrak{h}_n^{ \pm}(v)\right)\\
	&=r s\frac{u_{\pm}-v_{ \mp}}{r u_{\pm}-s v_{ \mp}} \left(\mathfrak{f}_{n+1, n-1}^{ \pm}(v) \mathfrak{h}_{n-1}^{ \pm}(v) \mathfrak{e}_{n-1, n}^{ \pm}(v)+\mathfrak{f}_{n+1, n}^{ \pm}(v) \mathfrak{h}_n^{ \pm}(v)\right) \mathfrak{h}_i^{\mp}(u),\\
	&\frac{u_{ \mp}-v_{\pm}}{r u_{ \mp}-s v_{\pm}}\mathfrak{h}_i^{\mp}(u) \mathfrak{f}_{n+1, n-1}^{ \pm}(v) \mathfrak{h}_{n-1}^{ \pm}(v)=r s\frac{u_{\pm}-v_{ \mp}}{r u_{\pm}-s v_{ \mp}} \mathfrak{f}_{n+1, n-1}^{ \pm}(v) \mathfrak{h}_{n-1}^{ \pm}(v) \mathfrak{h}_i^{\mp}(u),\\
	&\frac{u_{ \mp}-v_{\pm}}{r u_{ \mp}-s v_{\pm}}\mathfrak{h}_i^{\mp}(u) \mathfrak{h}_{n-1}^{ \pm}(v) \mathfrak{e}_{n-1, n}^{ \pm}(v)=r s \frac{u_{\pm}-v_{ \mp}}{r u_{\pm}-s v_{ \mp}} \mathfrak{h}_{n-1}^{ \pm}(v) \mathfrak{e}_{n-1, n}^{ \pm}(v) \mathfrak{h}_i^{\mp}(u).
\end{align*}
Combining with $\mathfrak{h}_i^{ \pm}(u) \mathfrak{h}_i^{\mp}(v)=\mathfrak{h}_i^{\mp}(v) \mathfrak{h}_i^{ \pm}(u)$ and $\frac{u_{ \pm}-v_{\mp}}{r u_{ \pm}-s v_{\mp}} \mathfrak{h}_i^{ \pm}(u) \mathfrak{h}_j^{\mp}(v)=\frac{u_{\mp}-v_{ \pm}}{r u_{\mp}-s v_{ \pm}} \mathfrak{h}_j^{\mp}(v) \mathfrak{h}_i^{ \pm}(u)$, we get the last relation.

The remaining cases of the type C relations involving $\mathcal{X}_i^{ \pm}(u)$ and $\mathfrak{h}_j^{ \pm}(v)$ can be proved using Remark \ref{r:22} and
Corollary \ref{c:15}. In particular, Remark \ref{r:22} and the corresponding type A relations imply that
\begin{flalign*}
\mathfrak{h}_{n^{\prime}}^{ \pm}(u)^{-1} \mathfrak{e}_{n^{\prime},(n-1)^{\prime}}^{ \pm}(v) \mathfrak{h}_{n^{\prime}}^{ \pm}(u)=\frac{s^{-1} u-r^{-1} v}{u-v} \mathfrak{e}_{n^{\prime},(n-1)^{\prime}}^{ \pm}(v)-\frac{\left(s^{-1}-r^{-1}\right) v}{u-v} \mathfrak{e}_{n^{\prime},(n-1)^{\prime}}^{ \pm}(u),
\end{flalign*}
which can be rewritten as
\begin{flalign*}
\mathfrak{h}_{n^{\prime}}^{ \pm}(u)^{-1} \mathfrak{e}_{n-1, n}^{ \pm}\left(v r^{-2} s^2\right) \mathfrak{h}_{n^{\prime}}^{ \pm}(u)=r^{-1} s^{-1}\left(\frac{r u-s v}{u-v} \mathfrak{e}_{n-1, n}^{ \pm}\left(v r^{-2} s^2\right)-\frac{(r-s) v}{u-v} \mathfrak{e}_{n-1, n}^{ \pm}\left(u r^{-2} s^2\right)\right),
\end{flalign*}
and then
\begin{flalign*}
	& \mathfrak{h}_{n+1}^{ \pm}(u)^{-1} \mathcal{X}_{n-1}^{+}(v) \mathfrak{h}_{n+1}^{ \pm}(u)=\frac{r^{-1} u_{\mp}-s^{-1} v}{r^{-1} s u_{\mp}-r s^{-1} v} \mathcal{X}_{n-1}^{ \pm}(v) ,\\
	& \mathfrak{h}_{n+1}^{ \pm}(u)^{-1} \mathcal{X}_{n-1}^{-}(v) \mathfrak{h}_{n+1}^{ \pm}(u)=\frac{r^{-1} u_{ \pm}-s^{-1} v}{r^{-1} s u_{ \pm}-r s^{-1} v} \mathcal{X}_{n-1}^{ \pm}(v).
\end{flalign*}
Now we consider the relations among $\mathcal{X}_i^{ \pm}(u)$. For $i<n-1$, using the same method as in Corollary \ref{c:15} we have that
\begin{flalign*}
\mathfrak{e}_{i, i+1}^{ \pm}(u) \mathfrak{e}_{n, n+1}^{ \pm}(v)=\mathfrak{e}_{n, n+1}^{ \pm}(v) \mathfrak{e}_{i, i+1}^{ \pm}(u), \quad \mathfrak{e}_{i, i+1}^{ \pm}(u) \mathfrak{e}_{n, n+1}^{\mp}(v)=\mathfrak{e}_{n, n+1}^{\mp}(v) \mathfrak{e}_{i, i+1}^{ \pm}(u),
\end{flalign*}
\begin{flalign*}
\mathfrak{f}_{i+1, i}^{ \pm}(u) \mathfrak{f}_{n+1, n}^{ \pm}(v)=\mathfrak{f}_{n+1, n}^{ \pm}(v) \mathfrak{f}_{i+1, i}^{ \pm}(u), \quad \mathfrak{f}_{i+1, i}^{ \pm}(u) \mathfrak{f}_{n+1, n}^{\mp}(v)=\mathfrak{f}_{n+1, n}^{\mp}(v) \mathfrak{f}_{i+1, i}^{ \pm}(u).
\end{flalign*}
Consequently for all $i<n-1$, $\mathcal{X}_i^{ \pm}(u) \mathcal{X}_n^{ \pm}(v)=\mathcal{X}_n^{ \pm}(v) \mathcal{X}_i^{ \pm}(u).$ For the relation
between $\mathcal{X}_{n-1}^{ \pm}(u)$ and $\mathcal{X}_n^{ \pm}(v)$, it is sufficient to see the case $n=2$. We have that
\begin{flalign}
\begin{aligned}\label{thm24.3}
&r s \frac{u_{ \pm}-v_{\mp}}{r u_{ \pm}-s v_{\mp}} \ell_{12}^{ \pm}(u) \ell_{23}^{\mp}(v)+\frac{(r-s) u_{ \pm}}{r u_{ \pm}-s v_{\mp}} \ell_{22}^{ \pm}(u) \ell_{13}^{\mp}(v)\\
&=\frac{\left(u_{\mp}-v_{ \pm}\right) s\left(u_{\mp}-r^{-2} s^2 v_{ \pm}\right)}{\left(r u_{\mp}-s v_{ \pm}\right)\left(u_{\mp}-r^{-3} s^3 v_{ \pm}\right)} \ell_{23}^{\mp}(v) \ell_{12}^{ \pm}(u)+\frac{(r-s) v_{ \pm}}{r u_{\mp}-s v_{ \pm}} \ell_{22}^{\mp}(v) \ell_{13}^{ \pm}(u)\\
&\quad+\frac{(r-s)\left(u_{\mp}-v_{ \pm}\right) v_{ \pm} r^{-\frac{3}{2}} s^{-\frac{3}{2}}}{\left(r u_{\mp}-s v_{ \pm}\right)\left(u_{\mp}-r^{-3} s^3 v_{ \pm}\right)} \ell_{21}^{\mp}(v) \ell_{14}^{ \pm}(u)+\frac{(r-s)\left(u_{\mp}-v_{ \pm}\right) v_{ \pm} r^{-2} s^2}{\left(r u_{\mp}-s v_{ \pm}\right)\left(u_{\mp}-r^{-3} s^3 v_{ \pm}\right)} \ell_{22}^{\mp}(v) \ell_{13}^{ \pm}(u)\\
&\quad-\frac{(r-s)\left(u_{\mp}-v_{ \pm}\right) r^{-\frac{1}{2}} s^{-\frac{1}{2}} u_{\mp}}{\left(r u_{\mp}-s v_{ \pm}\right)\left(u_{\mp}-r^{-3} s^3 v_{ \pm}\right)} \ell_{24}^{\mp}(v) \ell_{11}^{ \pm}(u) .
\end{aligned}
\end{flalign}
On the other hand, it follows from the Gauss decomposition that $\ell_{23}^{\mp}(v)=f_{21}^{\mp}(v) \mathfrak{h}_1^{\mp}(v) \mathfrak{e}_{13}^{\mp}(v)+\mathfrak{h}_2^{\mp}(v) \mathfrak{e}_{23}^{\mp}(v)$, which can be used to rewrite the left hand side of (\ref{thm24.3}) as
\begin{align*}
r s \frac{u_{ \pm}-v_{\mp}}{r u_{ \pm}-s v_{\mp}} \ell_{12}^{ \pm}(u) \mathfrak{h}_2^{\mp}(v) \mathfrak{e}_{23}^{\mp}(v)+r s \frac{u_{ \pm}-v_{\mp}}{r u_{ \pm}-s v_{\mp}} &\ell_{12}^{ \pm}(u) \mathfrak{f}_{21}^{\mp}(v) \mathfrak{h}_1^{\mp}(v) \mathfrak{e}_{13}^{\mp}(v)\\
&\quad +\frac{(r-s) u_{ \pm}}{r u_{ \pm}-s v_{\mp}} \ell_{22}^{ \pm}(u) \ell_{13}^{\mp}(v).
\end{align*}
By the defining relations between $\ell_{12}^{+}(u)$ and $\ell_{21}^{-}(v)$
\begin{flalign*}
r s \frac{u_{ \pm}-v_{\mp}}{r u_{ \pm}-s v_{\mp}} \ell_{12}^{ \pm}(u) \ell_{21}^{\mp}(v)+\frac{(r-s) u_{ \pm}}{r u_{ \pm}-s v_{\mp}} &\ell_{22}^{ \pm}(u) \ell_{11}^{\mp}(v)\\
&=\frac{u_{\mp}-v_{ \pm}}{r u_{\mp}-s v_{ \pm}} \ell_{21}^{\mp}(v) \ell_{12}^{ \pm}(u)+\frac{(r-s) u_{\mp}}{r u_{\mp}-s v_{ \pm}} \ell_{22}^{\mp}(v) \ell_{11}^{ \pm}(u),
\end{flalign*}
we can reduce the left-hand side of (\ref{thm24.3}) to the following form
\begin{equation}
\begin{aligned}\label{thm24.4}
r s \frac{u_{ \pm}-v_{\mp}}{r u_{ \pm}-s v_{\mp}} \ell_{12}^{ \pm}(u) \mathfrak{h}_2^{\mp}(v) \mathfrak{e}_{23}^{\mp}(v)+\frac{u_{\mp}-v_{ \pm}}{r u_{\mp}-s v_{ \pm}} &\ell_{21}^{\mp}(v) \ell_{12}^{ \pm}(u) \mathfrak{e}_{13}^{\mp}(v)
+\frac{(r-s) u_{\mp}}{r u_{\mp}-s v_{ \pm}} \ell_{22}^{\mp}(v) \ell_{11}^{ \pm}(u) \mathfrak{e}_{13}^{\mp}(v).
\end{aligned}
\end{equation}
Furthermore, by using the defining relations between $\ell_{12}^{ \pm}(u)$ and $\ell_{11}^{\mp}(v)$:
\begin{flalign*}
\ell_{12}^{ \pm}(u) \ell_{11}^{\mp}(v)=\frac{u_{\mp}-v_{ \pm}}{r u_{\mp}-s v_{ \pm}} \ell_{11}^{\mp}(v) \ell_{12}^{ \pm}(u)+\frac{(r-s) u_{\mp}}{r u_{\mp}-s v_{ \pm}} \ell_{12}^{\mp}(v) \ell_{11}^{ \pm}(u),
\end{flalign*}
we can finally write the left-hand side of (\ref{thm24.3}) as
\begin{flalign*}
r s \frac{u_{ \pm}-v_{\mp}}{r u_{ \pm}-s v_{\mp}} \ell_{12}^{ \pm}(u) \mathfrak{h}_2^{\mp}(v) \mathfrak{e}_{23}^{\mp}(v)+\mathfrak{f}_{21}^{\mp}(v) \ell_{12}^{ \pm}(u) \ell_{13}^{\mp}(v)+\frac{(r-s) u_{\mp}}{r u_{\mp}-s v_{ \pm}} \mathfrak{h}_2^{\mp}(v) \ell_{11}^{ \pm}(u) \mathfrak{e}_{13}^{\mp}(v) .
\end{flalign*}
On the other hand, applying the formula $\ell_{2 i}^{\mp}(v)$ for $i=2,3,4$, we write the right-hand side of (\ref{thm24.3}) as
\begin{equation}
\begin{aligned}\label{thm24.5}
&\frac{\left(u_{\mp}-v_{ \pm}\right) s\left(u_{\mp}-r^{-2} s^2 v_{ \pm}\right)}{\left(r u_{\mp}-s v_{ \pm}\right)\left(u_{\mp}-r^{-3} s^3 v_{ \pm}\right)} \mathfrak{h}_2^{\mp}(v) \mathfrak{e}_{23}^{\mp}(v) \ell_{12}^{ \pm}(u)+\frac{(r-s) v_{ \pm}}{r u_{\mp}-s v_{ \pm}} \mathfrak{h}_2^{\mp}(v) \ell_{13}^{ \pm}(u)\\
&+\frac{(r-s)\left(u_{\mp}-v_{ \pm}\right) v_{ \pm} r^{-2} s^2}{\left(r u_{\mp}-s v_{ \pm}\right)\left(u_{\mp}-r^{-3} s^3 v_{ \pm}\right)} \mathfrak{h}_2^{\mp}(v) \ell_{13}^{ \pm}(u)-\frac{(r-s)\left(u_{\mp}-v_{ \pm}\right) r^{-\frac{1}{2}} s^{-\frac{1}{2}} u_{\mp}}{\left(r u_{\mp}-s v_{ \pm}\right)\left(u_{\mp}-r^{-3} s^3 v_{ \pm}\right)} \mathfrak{h}_2^{\mp}(v) \mathfrak{e}_{24}^{\mp}(v) \ell_{11}^{ \pm}(u)\\
&\quad+\frac{\left(u_{\mp}-v_{ \pm}\right) s\left(u_{\mp}-r^{-2} s^2 v_{ \pm}\right)}{\left(r u_{\mp}-s v_{ \pm}\right)\left(u_{\mp}-r^{-3} s^3 v_{ \pm}\right)} f_{21}^{\mp}(v) \ell_{13}^{\mp}(v) \ell_{12}^{ \pm}(u)+\frac{(r-s) v_{ \pm}}{r u_{\mp}-s v_{ \pm}} f_{21}^{\mp}(v) \mathfrak{h}_1^{\mp}(v) \mathfrak{e}_{12}^{\mp}(v) \ell_{13}^{ \pm}(u)\\
&\quad +\frac{(r-s)\left(u_{\mp}-v_{ \pm}\right) v_{ \pm} r^{-2} s^2}{\left(r u_{\mp}-s v_{ \pm}\right)\left(u_{\mp}-r^{-3} s^3 v_{ \pm}\right)} f_{21}^{\mp}(v) \mathfrak{h}_1^{\mp}(v) \mathfrak{e}_{12}^{\mp}(v) \ell_{13}^{ \pm}(u)\\
&\quad -\frac{(r-s)\left(u_{\mp}-v_{ \pm}\right) r^{-\frac{1}{2}} s^{-\frac{1}{2}} u_{\mp}}{\left(r u_{\mp}-s v_{ \pm}\right)\left(u_{\mp}-r^{-3} s^3 v_{ \pm}\right)}  f_{21}^{\mp}(v) \mathfrak{h}_1^{\mp}(v) \mathfrak{e}_{14}^{\mp}(v) \ell_{11}^{ \pm}(u).
\end{aligned}
\end{equation}
It follows from the relations between $\ell_{12}^{ \pm}(u)$ and $\ell_{13}^{\mp}(v)$ that
\begin{align*}
&\ell_{12}^{ \pm}(u) \ell_{13}^{\mp}(v)=\frac{(r-s)\left(u_{\mp}-v_{ \pm}\right) v_{ \pm} r^{-\frac{3}{2}} s^{-\frac{3}{2}}}{\left(r u_{\mp}-s v_{ \pm}\right)\left(u_{\mp}-r^{-3} s^3 v_{ \pm}\right)} \ell_{11}^{\mp}(v) \ell_{14}^{ \pm}(u)\\
&\quad +\left(\frac{(r-s)\left(u_{\mp}-v_{ \pm}\right) v_{ \pm} r^{-2} s^2}{\left(r u_{\mp}-s v_{ \pm}\right)\left(u_{\mp}-r^{-3} s^3 v_{ \pm}\right)}+\frac{(r-s) v_{ \pm}}{r u_{\mp}-s v_{ \pm}}\right) \ell_{12}^{\mp}(v) \ell_{13}^{ \pm}(u)\\
&\quad+\frac{\left(u_{\mp}-v_{ \pm}\right) s\left(u_{\mp}-r^{-2} s^2 v_{ \pm}\right)}{\left(r u_{\mp}-s v_{ \pm}\right)\left(u_{\mp}-r^{-3} s^3 v_{ \pm}\right)} \ell_{13}^{\mp}(v) \ell_{12}^{ \pm}(u)+\frac{(s-r)\left(u_{\mp}-v_{ \pm}\right) r^{-\frac{1}{2}} s^{-\frac{1}{2}} u_{\mp}}{\left(r u_{\mp}-s v_{ \pm}\right)\left(u_{\mp}-r^{-3} s^3 v_{ \pm}\right)} \ell_{14}^{\mp}(v) \ell_{11}^{ \pm}(u).
\end{align*}
Comparing the equations (\ref{thm24.4}) and (\ref{thm24.5}), we can rewrite
\begin{align*}
&r s \frac{u_{ \pm}-v_{\mp}}{r u_{ \pm}-s v_{\mp}} \ell_{12}^{ \pm}(u) \mathfrak{h}_2^{\mp}(v) \mathfrak{e}_{23}^{\mp}(v)+\frac{(r-s) u_{\mp}}{r u_{\mp}-s v_{ \pm}} \mathfrak{h}_2^{\mp}(v) \ell_{11}^{ \pm}(u) \mathfrak{e}_{13}^{\mp}(v)\\
&=\frac{\left(u_{\mp}-v_{ \pm}\right) s\left(u_{\mp}-r^{-2} s^2 v_{ \pm}\right)}{\left(r u_{\mp}-s v_{ \pm}\right)\left(u_{\mp}-r^{-3} s^3 v_{ \pm}\right)} \mathfrak{h}_2^{\mp}(v) \mathfrak{e}_{23}^{\mp}(v) \ell_{12}^{ \pm}(u)+\frac{(r-s) v_{ \pm}}{r u_{\mp}-s v_{ \pm}} \mathfrak{h}_2^{\mp}(v) \ell_{13}^{ \pm}(u)\\
&\quad +\frac{(r-s)\left(u_{\mp}-v_{ \pm}\right) v_{ \pm} r^{-2} s^2}{\left(r u_{\mp}-s v_{ \pm}\right)\left(u_{\mp}-r^{-3} s^3 v_{ \pm}\right)} \mathfrak{h}_2^{\mp}(v) \ell_{13}^{ \pm}(u)\\
&-\frac{(r-s)\left(u_{\mp}-v_{ \pm}\right) r^{-\frac{1}{2}} s^{-\frac{1}{2}} u_{\mp}}{\left(r u_{\mp}-s v_{ \pm}\right)\left(u_{\mp}-r^{-3} s^3 v_{ \pm}\right)} \mathfrak{h}_2^{\mp}(v) \mathfrak{e}_{24}^{\mp}(v) \ell_{11}^{ \pm}(u).
\end{align*}
Recalling Proposition \ref{p:20}, we have that
\begin{flalign*}
\mathfrak{e}_{m j}^{ \pm}(u) \ell_{k l}^{\mp[n-m]}(v)=r^{-1} s^{-1} \frac{r u_{\mp}-s v_{ \pm}}{u_{\mp}-v_{ \pm}} \ell_{k j}^{\mp[n-m]}(v) \mathfrak{e}_{m l}^{ \pm}(u)-r^{-1} s^{-1} \frac{(r-s) u_{\mp}}{u_{\mp}-v_{ \pm}} \ell_{k j}^{\mp[n-m]}(v) \mathfrak{e}_{m j}^{\mp}(v),
\end{flalign*}
\begin{flalign*}
\frac{u_{ \pm}-v_{\mp}}{r u_{ \pm}-s v_{\mp}} \mathfrak{h}_1^{ \pm}(u) \mathfrak{h}_2^{\mp}(v)=\frac{u_{\mp}-v_{ \pm}}{r u_{\mp}-s v_{ \pm}} \mathfrak{h}_2^{\mp}(v) \mathfrak{h}_1^{ \pm}(u),
\end{flalign*}
then we get the relation
\begin{flalign*}
\ell_{12}^{ \pm}(u) \mathfrak{h}_2^{\mp}(v)=r^{-1} s^{-1} \frac{r u_{\mp}-s v_{ \pm}}{u_{\mp}-v_{ \pm}} \mathfrak{h}_1^{ \pm}(u) \mathfrak{h}_2^{\mp}(v) \mathfrak{e}_{12}^{ \pm}(u)-r^{-1} s^{-1} \frac{(r-s) u_{\mp}}{u_{\mp}-v_{ \pm}} \mathfrak{h}_1^{ \pm}(u) \mathfrak{h}_2^{\mp}(v) \mathfrak{e}_{12}^{\mp}(v)
\end{flalign*}
\begin{flalign*}
=r^{-1} s^{-1} \frac{r u_{ \pm}-s v_{\mp}}{u_{ \pm}-v_{\mp}} \mathfrak{h}_2^{\mp}(v) \mathfrak{h}_1^{ \pm}(u) \mathfrak{e}_{12}^{ \pm}(u)-r^{-1} s^{-1} \frac{(r-s) u_{\mp}}{r u_{\mp}-s v_{ \pm}} \frac{r u_{ \pm}-s v_{\mp}}{u_{ \pm}-v_{\mp}} \mathfrak{h}_2^{\mp}(v) \mathfrak{h}_1^{ \pm}(u) \mathfrak{e}_{12}^{\mp}(v) .
\end{flalign*}
Moreover, using Corollary \ref{c:15}, we also have:
\begin{flalign*}
r s \mathfrak{h}_1^{ \pm}(u) \mathfrak{e}_{23}^{\mp}(v)=\mathfrak{e}_{23}^{\mp}(v) \mathfrak{h}_1^{ \pm}(u).
\end{flalign*}
Dividing $\mathfrak{h}_1^{ \pm}(u)$ and $\mathfrak{h}_2^{\mp}(v)$, we get the following relation
\begin{flalign*}
\begin{aligned}
	& \mathfrak{e}_{12}^{ \pm}(u) \mathfrak{e}_{23}^{\mp}(v)+\frac{(r-s) u_{\mp}}{r u_{\mp}-s v_{ \pm}} \mathfrak{e}_{13}^{\mp}(v)-\frac{(r-s) u_{\mp}}{r u_{\mp}-s v_{ \pm}} \mathfrak{e}_{12}^{\mp}(v) \mathfrak{e}_{23}^{\mp}(v) \\
	& =r s \frac{\left(u_{\mp}-v_{ \pm}\right) s\left(u_{\mp}-r^{-2} s^2 v_{ \pm}\right)}{\left(r u_{\mp}-s v_{ \pm}\right)\left(u_{\mp}-r^{-3} s^3 v_{ \pm}\right)} \mathfrak{e}_{23}^{\mp}(v) \mathfrak{e}_{12}^{ \pm}(u)+\frac{(r-s) v_{ \pm}}{r u_{\mp}-s v_{ \pm}} \mathfrak{e}_{13}^{ \pm}(u) \\
	& +\frac{(r-s)\left(u_{\mp}-v_{ \pm}\right) v_{ \pm} r^{-2} s^2}{\left(r u_{\mp}-s v_{ \pm}\right)\left(u_{\mp}-r^{-3} s^3 v_{ \pm}\right)} \mathfrak{e}_{13}^{ \pm}(u)-\frac{(r-s)\left(u_{\mp}-v_{ \pm}\right) r^{-\frac{1}{2}} s^{-\frac{1}{2}} u_{\mp}}{\left(r u_{\mp}-s v_{ \pm}\right)\left(u_{\mp}-r^{-3} s^3 v_{ \pm}\right)} \mathfrak{h}_1^{ \pm}(u)^{-1} \mathfrak{e}_{24}^{\mp}(v) \mathfrak{h}_1^{ \pm}(u).
\end{aligned}
\end{flalign*}
Similar arguments imply following relations
\begin{flalign*}
\begin{aligned}
	& r s \frac{u_{ \pm}-v_{\mp}}{r u_{ \pm}-s v_{\mp}} \mathfrak{h}_1^{ \pm}(u) \mathfrak{h}_2^{\mp}(v) \mathfrak{e}_{24}^{\mp}(v)=\frac{\left(u_{\mp}-v_{ \pm}\right) s\left(u_{\mp}-r^{-2} s^2 v_{ \pm}\right)}{\left(r u_{\mp}-s v_{ \pm}\right)\left(u_{\mp}-r^{-3} s^3 v_{ \pm}\right)} \mathfrak{h}_2^{\mp}(v) \mathfrak{e}_{24}^{\mp}(v) \mathfrak{h}_1^{ \pm}(u) \\
	& +\frac{(r-s)\left(u_{\mp}-v_{ \pm}\right) v_{ \pm} r^{-\frac{3}{2}} s^{-\frac{3}{2}}}{\left(r u_{\mp}-s v_{ \pm}\right)\left(u_{\mp}-r^{-3} s^3 v_{ \pm}\right)} \mathfrak{h}_2^{\mp}(v) \mathfrak{h}_1^{ \pm}(u) \mathfrak{e}_{13}^{ \pm}(u) \\
	& -\frac{(r-s)\left(u_{\mp}-v_{ \pm}\right) v_{ \pm} r^{-\frac{5}{2}} s^{\frac{5}{2}}}{\left(r u_{\mp}-s v_{ \pm}\right)\left(u_{\mp}-r^{-3} s^3 v_{ \pm}\right)} \mathfrak{h}_2^{\mp}(v) \mathfrak{e}_{23}^{\mp}(v) \mathfrak{h}_1^{ \pm}(u) \mathfrak{e}_{12}^{ \pm}(u) .
\end{aligned}
\end{flalign*}
Since
\begin{flalign*}
\frac{u_{ \pm}-v_{\mp}}{r u_{ \pm}-s v_{\mp}} \mathfrak{h}_1^{ \pm}(u) \mathfrak{h}_2^{\mp}(v)=\frac{u_{\mp}-v_{ \pm}}{r u_{\mp}-s v_{ \pm}} \mathfrak{h}_2^{\mp}(v) \mathfrak{h}_1^{ \pm}(u),
\end{flalign*}
we then get following relation
\begin{flalign*}
\begin{aligned}
	& r s \mathfrak{h}_1^{ \pm}(u) \mathfrak{e}_{24}^{\mp}(v)=\frac{s\left(u_{\mp}-v_{ \pm} r^{-2} s^2\right)}{u_{\mp}-r^{-3} s^3 v_{ \pm}} \mathfrak{e}_{24}^{\mp}(v) \mathfrak{h}_1^{ \pm}(u) \\
	& +\frac{(r-s) v_{ \pm} r^{-\frac{3}{2}} s^{-\frac{3}{2}}}{u_{\mp}-r^{-3} s^3 v_{ \pm}} \mathfrak{h}_1^{ \pm}(u) \mathfrak{e}_{13}^{ \pm}(u)-\frac{(r-s) v_{ \pm} r^{-\frac{5}{2}} s^{\frac{5}{2}}}{u_{\mp}-r^{-3} s^3 v_{ \pm}} \mathfrak{e}_{23}^{\mp}(v) \mathfrak{h}_1^{ \pm}(u) \mathfrak{e}_{12}^{ \pm}(u),
\end{aligned}
\end{flalign*}
combining with the above relations, we get that
\begin{flalign*}
\begin{aligned}
	& \mathfrak{e}_{12}^{ \pm}(u) \mathfrak{e}_{23}^{\mp}(v)=\frac{1}{r s^{-1} u_{\mp}-r^{-1} s v_{ \pm}}\left(\left(r s^{-1}-r^{-1} s\right) v_{ \pm} \mathfrak{e}_{13}^{ \pm}(u)+r s\left(u_{\mp}-v_{ \pm}\right) \mathfrak{e}_{23}^{\mp}(v) \mathfrak{e}_{12}^{ \pm}(u)\right. \\
	& \left.-\left(r s^{-1}-r^{-1} s\right) u_{\mp} \mathfrak{e}_{12}^{\mp}(v) \mathfrak{e}_{23}^{\mp}(v)+\left(r s^{-1}-r^{-1} s\right) u_{\mp} \mathfrak{e}_{13}^{\mp}(v)\right) \\
	& +\frac{(r-s) u_{\mp}\left(u_{\mp}-v_{ \pm}\right)}{\left(r s^{-1} u_{\mp}-r^{-1} s v_{ \pm}\right)\left(r u_{\mp}-s v_{ \pm}\right)}\left(\mathfrak{e}_{12}^{\mp}(v) \mathfrak{e}_{23}^{\mp}(v)-\mathfrak{e}_{13}^{\mp}(v)-r s^{-1} \mathfrak{e}_{24}^{\mp}(v)\right) .
\end{aligned}
\end{flalign*}
Multiply both sides by $r u_{\mp}-s v_{ \pm}$ and setting $r u_{\mp}=s v_{ \pm}$ we see that the second summand vanishes. Finally, by Theorem \ref{t:homo},
we have that
\begin{flalign*}
\begin{aligned}
	& \left(r s^{-1} u_{\mp}-r^{-1} s v_{ \pm}\right) \mathfrak{e}_{n-1, n}^{ \pm}(u) \mathfrak{e}_{n, n+1}^{\mp}(v) \\
	& =r s\left(u_{\mp}-v_{ \pm}\right) \mathfrak{e}_{n, n+1}^{\mp}(v) \mathfrak{e}_{n-1, n}^{ \pm}(u)+\left(r s^{-1}-r^{-1} s\right) v_{ \pm} \mathfrak{e}_{n-1, n+1}^{ \pm}(u) \\
	& -\left(r s^{-1}-r^{-1} s\right) u_{\mp} \mathfrak{e}_{n-1, n}^{\mp}(v) \mathfrak{e}_{n, n+1}^{\mp}(v)+\left(r s^{-1}-r^{-1} s\right) u_{\mp} \mathfrak{e}_{n-1, n+1}^{\mp}(v) .
\end{aligned}
\end{flalign*}
By similar consideration, we have the following relations
\begin{align*}
	& \left(r s^{-1} u-r^{-1} s v\right) \mathfrak{e}_{n-1, n}^{ \pm}(u) \mathfrak{e}_{n, n+1}^{ \pm}(v) \\
	& =r s(u-v) \mathfrak{e}_{n, n+1}^{ \pm}(v) \mathfrak{e}_{n-1, n}^{ \pm}(u)+\left(r s^{-1}-r^{-1} s\right) v \mathfrak{e}_{n-1, n+1}^{ \pm}(u) \\
	& -\left(r s^{-1}-r^{-1} s\right) u \mathfrak{e}_{n-1, n}^{ \pm}(v) \mathfrak{e}_{n, n+1}^{ \pm}(v)+\left(r s^{-1}-r^{-1} s\right) u \mathfrak{e}_{n-1, n+1}^{ \pm}(v) ,
\end{align*}
and
\begin{align*}
	& r s\left(u_{ \pm}-v_{\mp}\right) \mathfrak{f}_{n, n-1}^{ \pm}(u) \mathfrak{f}_{n+1, n}^{\mp}(v) \\
	& =\left(r s^{-1} u_{ \pm}-r^{-1} s v_{\mp}\right) \mathfrak{f}_{n+1, n}^{\mp}(v) \mathfrak{f}_{n, n-1}^{ \pm}(u)-\left(r s^{-1}-r^{-1} s\right) v_{\mp} \mathfrak{f}_{n+1, n-1}^{\mp}(v) \\
	& +\left(r s^{-1}-r^{-1} s\right) v_{\mp} \mathfrak{f}_{n+1, n}^{\mp}(v) \mathfrak{f}_{n, n-1}^{\mp}(v)+\left(r s^{-1}-r^{-1} s\right) u_{ \pm} \mathfrak{f}_{n+1, n-1}^{ \pm}(u),
\end{align*}
as well as
\begin{align*}
	& r s(u-v) \mathfrak{f}_{n, n-1}^{ \pm}(u) \mathfrak{f}_{n+1, n}^{ \pm}(v) \\
	& =\left(r s^{-1} u-r^{-1} s v\right) \mathfrak{f}_{n+1, n}^{ \pm}(v) \mathfrak{f}_{n, n-1}^{ \pm}(u)-\left(r s^{-1}-r^{-1} s\right) v \mathfrak{f}_{n+1, n-1}^{ \pm}(v) \\
	& +\left(r s^{-1}-r^{-1} s\right) v \mathfrak{f}_{n+1, n}^{ \pm}(v) \mathfrak{f}_{n, n-1}^{ \pm}(v)+\left(r s^{-1}-r^{-1} s\right) u \mathfrak{f}_{n+1, n-1}^{ \pm}(u) .
\end{align*}
Therefore we have shown that 
\begin{flalign*}
\left(u r s^{-1}-r^{-1} s v\right)^{ \pm 1} \mathcal{X}_{n-1}^{ \pm}(u) \mathcal{X}_n^{ \pm}(v)=(r s)^{ \pm 1}(u-v)^{ \pm 1} \mathcal{X}_n^{ \pm}(v) \mathcal{X}_{n-1}^{ \pm}(u).
\end{flalign*}
Now for $i \leqslant n-1$, we would like to show that 
\begin{flalign*}
\begin{array}{lll}
	\mathfrak{e}_{i, i+1}^{ \pm}(u) \mathfrak{f}_{n+1, n}^{ \pm}(v)=\mathfrak{f}_{n+1, n}^{ \pm}(v) \mathfrak{e}_{i, i+1}^{ \pm}(u), & & \mathfrak{e}_{i, i+1}^{ \pm}(u) \mathfrak{f}_{n+1, n}^{\mp}(v)=\mathfrak{f}_{n+1, n}^{\mp}(v) \mathfrak{e}_{i, i+1}^{ \pm}(u), \\
	\mathfrak{f}_{i+1, i}^{ \pm}(u) \mathfrak{e}_{n, n+1}^{ \pm}(v)=\mathfrak{e}_{n, n+1}^{ \pm}(v) \mathfrak{f}_{i+1, i}^{ \pm}(u), & & \mathfrak{f}_{i+1, i}^{ \pm}(u) \mathfrak{e}_{n, n+1}^{\mp}(v)=\mathfrak{e}_{n, n+1}^{\mp}(v) \mathfrak{f}_{i+1, i}^{ \pm}(u) .
\end{array}
\end{flalign*}
The proof of these four relations are quite similar, so we only verify the second one. For $i \leqslant n-2$, Corollary \ref{c:15} gives that
\begin{align*}
\frac{u_{ \pm}-v_{\mp}}{r u_{ \pm}-s v_{\mp}} \mathfrak{h}_i^{ \pm}(u) \mathfrak{e}_{i, i+1}^{ \pm}(u) \mathfrak{f}_{n+1, n}^{\mp}(v) \mathfrak{h}_n^{\mp}(v)&=rs\frac{u_{\mp}-v_{ \pm}}{r u_{\mp}-s v_{ \pm}} \mathfrak{f}_{n+1, n}^{\mp}(v) \mathfrak{h}_n^{\mp}(v) \mathfrak{h}_i^{ \pm}(u) \mathfrak{e}_{i, i+1}^{ \pm}(u),\\
\frac{u_{ \pm}-v_{\mp}}{r u_{ \pm}-s v_{\mp}} \mathfrak{h}_i^{ \pm}(u) \mathfrak{f}_{n+1, n}^{\mp}(v) \mathfrak{h}_n^{\mp}(v)&=rs\frac{u_{\mp}-v_{ \pm}}{r u_{\mp}-s v_{ \pm}} \mathfrak{f}_{n+1, n}^{\mp}(v) \mathfrak{h}_n^{\mp}(v) \mathfrak{h}_i^{ \pm}(u) .
\end{align*}
Therefore,
\begin{flalign*}
\mathfrak{e}_{i, i+1}^{ \pm}(u) \mathfrak{f}_{n+1, n}^{\mp}(v) \mathfrak{h}_n^{\mp}(v)=\mathfrak{f}_{n+1, n}^{\mp}(v) \mathfrak{h}_n^{\mp}(v) \mathfrak{e}_{i, i+1}^{ \pm}(u).
\end{flalign*}
Also by Corollary \ref{c:15}, $[\mathfrak{h}_n^{\mp}(v), \mathfrak{e}_{i, i+1}^{ \pm}(u)]=0$, which implies that
$\mathfrak{e}_{i, i+1}^{ \pm}(u) \mathfrak{f}_{n+1, n}^{\mp}(v)=\mathfrak{f}_{n+1, n}^{\mp}(v) \mathfrak{e}_{i, i+1}^{ \pm}(u).$

By Proposition \ref{p:20}, we have that
\begin{flalign*}
\begin{aligned}
	& \mathfrak{e}_{n-1, n}^{ \pm}(u) \mathfrak{f}_{n+1, n}^{\mp}(v) \mathfrak{h}_n^{\mp}(v)=r^{-1} s^{-1} \frac{r u_{\mp}-s v_{ \pm}}{u_{\mp}-v_{ \pm}} \mathfrak{f}_{n+1, n}^{\mp}(v) \mathfrak{h}_n^{\mp}(v) \mathfrak{e}_{n-1, n}^{ \pm}(u) \\
	& +r^{-1} s^{-1} \frac{(r-s) v_{ \pm}}{u_{\mp}-v_{ \pm}} \mathfrak{f}_{n+1, n}^{\mp}(v) \mathfrak{h}_n^{\mp}(v) \mathfrak{e}_{n-1, n}^{\mp}(v),
\end{aligned}
\end{flalign*}
and
\begin{flalign*}
\mathfrak{e}_{n-1, n}^{ \pm}(u) \mathfrak{h}_n^{\mp}(v)=r^{-1} s^{-1} \frac{r u_{\mp}-s v_{ \pm}}{u_{\mp}-v_{ \pm}} \mathfrak{h}_n^{\mp}(v) \mathfrak{e}_{n-1, n}^{ \pm}(u)+r^{-1} s^{-1} \frac{(r-s) v_{ \pm}}{u_{\mp}-v_{ \pm}} \mathfrak{h}_n^{\mp}(v) \mathfrak{e}_{n-1, n}^{\mp}(v) .
\end{flalign*}
Therefore, for $i \leqslant n-1$,
\begin{flalign*}
\mathfrak{e}_{i, i+1}^{ \pm}(u) \mathfrak{f}_{n+1, n}^{\mp}(v)=\mathfrak{f}_{n+1, n}^{\mp}(v) \mathfrak{e}_{i, i+1}^{ \pm}(u).
\end{flalign*}
We have now derived all commutation relations for $\mathcal{X}_i^{+}(u)$ and $\mathcal{X}_j^{-}(v)$.
\end{proof}
We have shown that the following map $\tau: U_{r, s}(\mathrm{C}_n^{(1)}) \rightarrow U(\bar{R})$ is a homomorphism.
\begin{flalign*}
\begin{aligned}
	& x_i^{ \pm}(u)\mapsto (r_i-s_i)^{-1} X_i^{ \pm}(u(r s^{-1})^{\frac{i}{2}}), \\
	& \psi_i(u)\mapsto r^{-1}s^{-1}h_{i+1}^{-}(u(r s^{-1})^{\frac{i}{2}}) h_i^{-}(u(r s^{-1})^{\frac{i}{2}})^{-1}, \\
	& \varphi_i(u)\mapsto r^{-1}s^{-1}h_{i+1}^{+}(u(r s^{-1})^{\frac{i}{2}}) h_i^{+}(u(r s^{-1})^{\frac{i}{2}})^{-1},
\end{aligned}
\end{flalign*}
for $i=1, \ldots, n-1$, and
\begin{flalign*}
\begin{aligned}
	& x_n^{ \pm}(u)\mapsto (r_n-s_n)^{-1} X_n^{ \pm}(u(r s^{-1})^{\frac{n+1}{2}}), \\
	& \psi_n(u)\mapsto r^{-2}s^{-2}h_{n+1}^{-}(u(r s^{-1})^{\frac{n+1}{2}}) h_n^{-}(u(r s^{-1})^{\frac{n+1}{2}})^{-1}, \\
	& \varphi_n(u)\mapsto r^{-2}s^{-2}h_{n+1}^{+}(u(r s^{-1})^{\frac{n+1}{2}}) h_n^{+}(u(r s^{-1})^{\frac{n+1}{2}})^{-1} .
\end{aligned}
\end{flalign*}
\begin{remark} The Drinfeld realization and FRT formulism of the two-paremater quantum enveloping algebra
	$U_{r, s}(\widehat{\mathfrak g})$ in types $B_n^{(1)}$, $D_n^{(1)}$ have been recently given in \cite{HXZ, ZHX}.
\end{remark}

\bigskip

{\bf Statement of Confict Interest}. On behalf of all authors, the corresponding author states that
there is no conflict of interest.

\bigskip

{\bf Data Availability Statement}. All data generated during the study are included in the article.

\end{document}